\documentclass[11pt]{article}
\usepackage{amsmath,amsthm,amssymb}
\usepackage[T1]{fontenc}
\usepackage[utf8]{inputenc}
\usepackage[dvipdf]{graphicx}
\usepackage{color}
\usepackage{epstopdf}
\usepackage{dsfont}
\usepackage{mathtools}
\mathtoolsset{showonlyrefs=true}
\usepackage{tikz}
\usepackage{enumerate}
\usepackage{mathrsfs}
\usepackage[normalem]{ulem}
\usepackage[colorlinks=true,citecolor=red,linkcolor=db,urlcolor=blue,pdfstartview=FitH]{hyperref}
\usepackage[numbers]{natbib}
\usepackage{comment}

\definecolor{db}{RGB}{0, 0, 130}
\definecolor{rp}{rgb}{0.25, 0, 0.75}
\definecolor{dg}{rgb}{0, 0.6, 0}

\newtheorem{theorem}{Theorem}[section]

\newtheorem{definition}{Definition}[section]
\newtheorem{prop}[definition]{Proposition}
\newtheorem{corollary}[definition]{Corollary}
\newtheorem{example}[definition]{Example}
\newtheorem{assumption}[definition]{Assumption}
\newtheorem{lemma}[definition]{Lemma}

\newtheorem{proposition}[definition]{Proposition}
\newtheorem{remark}[definition]{Remark}

\def\Ac{\mathcal A}

\def\Cc{\mathcal C}
\def\Fc{\mathcal F}
\def\Gc{\mathcal G}

\def\Lc{\mathcal L}
\def\Nc{\mathcal N}
\def\Pc{\mathcal P}

\def\Wc{\mathcal W}

\def\E{\mathbb{E}}

\def\N{\mathbb N}
\def\P{\mathbb P}
\def\R{\mathbb R}

\def\x{\times}

\def\eps{\varepsilon}
\def\dd{\,\mathrm{d}}
\def\1{\mathbf{1}}

\newcommand{\norm}[1]{\left\lVert #1\right\rVert}
\newcommand{\abs}[1]{\left\lvert #1\right\rvert}
\DeclareMathOperator*{\esssup}{ess\,sup}
\title{Sharp functional quantization and empirical Wasserstein rates for It\^o processes}
\author{ Xihao He
        \footnote{Department of Mathematics, University of Southern California. xihaohe@usc.edu}
        \and 
        Fengyi Yuan
        \footnote{Corresponding author. School of Science and Engineering, the Chinese University of Hong Kong (Shenzhen). yuanfengyi@cuhk.edu.cn.
        Fengyi Yuan is supported by the Chinese University of Hong Kong (Shenzhen) start-up fund under UDF01004253.}
        }

\date{}

\begin{document}
\maketitle

\begin{abstract}
We establish functional quantization and empirical Wasserstein rates for
continuous It\^o processes under the supremum norm. We assume that the
initial condition and the drift and diffusion integrands are controlled by
a time-uniform random upper bound with a finite $\rho$-moment for some
$\rho>1$. Under this assumption, the $n$-point $L^q$ quantization error is
at most $C(\log n)^{-1/2}$ for every $1\leq q<\rho$. The known Brownian
lower bound shows that the exponent $1/2$ is optimal over this class.
Our proof combines an adaptive dyadic time partition with localization
according to the size of the integrands.
A general transfer principle yields the sharp mean rate
$(\log N)^{-1/2}$ for the $p$-Wasserstein distance between the empirical law
of $N$ independent copies and their common path law, together with
nonasymptotic deviation bounds, whenever $1\leq p<\rho$. Applications
include empirical path-law estimates for path-dependent SDEs and a path-space
limit-theory estimate for path-dependent McKean--Vlasov interacting
particle systems with common noise.
\end{abstract}

\noindent\textbf{MSC2010 subject classification:}
Primary 60B10; Secondary 49Q22, 60H10, 60K35.


\tableofcontents
\section{Introduction}

Let $X$ be a continuous $\R^d$-valued It\^o process on $[0,T]$, let
$\mu:=\Lc(X)$ be its law on $C([0,T];\R^d)$ equipped with
$d_\infty(x,y):=\|x-y\|_\infty$, and let
\begin{equation}
    \mu_N:=\frac1N\sum_{i=1}^N\delta_{X_i}
\end{equation}
be the empirical law of $N$ independent copies of $X$. We study both the
approximation of $X$ by deterministic finite codebooks and the convergence
of $\mu_N$ to $\mu$ in the $p$-Wasserstein distance associated with
$d_\infty$. For Wiener measure, the empirical-measure estimates of
Boissard and Le Gouic \cite{BoissardLeGouic2014} and the quantization
asymptotics of Dereich and Scheutzow \cite{DereichScheutzow2006} identify the
logarithmic exponent $1/2$ in both problems. This raises two questions. Does
the same exponent persist for functional quantization and empirical
approximation of general It\^o processes? Can one also obtain quantitative
deviations or concentration for the distance $\Wc_p(\mu_N,\mu)$?

We give an affirmative answer to both questions under a general assumption.
Assumption~\ref{assum:regularity} allows a random initial condition and general
progressively measurable drift and diffusion integrands. The initial condition
and both integrands are controlled by a time-uniform random upper bound
$\Lambda_X$ satisfying $\E\Lambda_X^\rho<\infty$ for some $\rho>1$.
Our first main result, Theorem~\ref{thm:functional-quantization}, establishes
the functional quantization bound
\begin{equation}
    e_{n,q}(X)\leq C(\log n)^{-1/2},
    \qquad n\geq2,\quad 1\leq q<\rho,
\end{equation}
where $e_{n,q}(X)$ is the smallest $L^q$ approximation error over
deterministic codebooks containing at most $n$ paths.
The known Brownian quantization lower bound shows that the exponent $1/2$
is optimal over the class in Assumption~\ref{assum:regularity}.

For $1\leq p<\rho$, Theorem~\ref{thm:endpoint-path-law} then establishes
the empirical rate
\begin{equation}
    \E\big[\Wc_p(\mu_N,\mu)\big]
    \leq C(\log N)^{-1/2},
    \qquad N\geq2.
\end{equation}
The same Brownian lower bound gives the matching worst-case order for
empirical approximation, as recorded in
Remark~\ref{remark:uniform-optimality}.

The moment assumption also yields a nonasymptotic deviation estimate.
Theorem~\ref{cor:polynomial-path-law-concentration} separates an exponential
term for the truncated empirical law from a polynomial remainder determined by
the available $\rho$-moment. Under the additional assumption that the path law
satisfies a quadratic transport inequality,
Proposition~\ref{prop:transport-concentration} gives exponential concentration
around the same logarithmic mean scale. 

We give two applications of the path-law estimate. First,
Proposition~\ref{prop:SDE} treats path-dependent SDEs with jointly Borel
measurable coefficients of linear growth and an initial condition in
$L^\rho$ for some $\rho>p\geq1$, whenever a continuous adapted solution exists.
The diffusion coefficient may be unbounded or degenerate.
Second, Proposition~\ref{thm:app} combines the empirical estimate for
conditionally independent copies of a McKean--Vlasov limit process with synchronous coupling. It
gives the same $(\log N)^{-1/2}$ mean rate for path-space limit theory
in a path-dependent mean-field particle system. We only require the coefficients to be
globally Lipschitz in the stopped path and the path law, and are uniformly
bounded at the origin. The initial distribution is assumed to have an $r$-moment for some $r>p$, as
specified in Assumption~\ref{ass:coefficients}.

The proofs of our main results contain three key components: adaptive
coding for processes with uniformly bounded initial value and semimartingale
characteristics, localization according
to dyadic ranges of $\Lambda_X$, and a transfer from functional quantization to
empirical measures.

The first ingredient, Proposition~\ref{prop:adaptive-dyadic-coding}, gives
the codebook construction for processes with bounded characteristics. Both the result and the proof seem to be new to the literature to the best of our knowledge, and are interesting in their own right. It
uses a dyadic partition of $[0,T]$ selected according to the realized path
oscillations. Suppose that a continuous semimartingale $Y$ has
its initial value, drift, and diffusion integrand uniformly bounded by a deterministic
constant $G$. For $0<\delta\leq G$, the construction recursively bisects a time
interval only when the realized path deviates from its value at the left
endpoint by more than $\delta$. Conditional exponential martingale estimates
show that deep refinements are sufficiently rare. Encoding the admissible tree
shapes and quantized endpoint increments then produces a deterministic
codebook $\Ac_\delta$ satisfying
\begin{equation}
    \log|\Ac_\delta|
    \leq C\frac{G^2}{\delta^2}.
\end{equation}
We also construct an event $\Gc_\delta$ by controlling the number of dyadic intervals, such that
\begin{equation}
    \P(\Gc_\delta^c)\leq \exp(-cG^2/\delta^2),
    \qquad
    d_\infty(Y,\Ac_\delta)\leq C\delta
    \quad\text{on }\Gc_\delta.
\end{equation}
A moment estimate controls the exceptional event, and
the choice $\delta\asymp G(\log n)^{-1/2}$ gives
$e_{n,q}(Y)\leq C_{d,q,T}G(\log n)^{-1/2}$ for every $q\geq1$, as recorded
in Lemma~\ref{lem:bounded-functional-quantization}.

The second ingredient removes the deterministic bound $G$.
Proposition~\ref{prop:endpoint-functional-quantization} partitions the
probability space according to dyadic ranges of $\Lambda_X$. On each resulting
event, we radially truncate the initial condition and the drift and diffusion
integrands. The
resulting process agrees with $X$ on that event and satisfies the hypotheses of
Lemma~\ref{lem:bounded-functional-quantization}. A decreasing allocation of codepoints keeps the combined
codebook within the prescribed cardinality. The strict gap $q<\rho$ makes the
errors over these events summable and controls the remaining contribution from
large values of $\Lambda_X$, preserving the
exponent $1/2$ under our general assumption. For comparison, standard H\"older-ball localization arguments for empirical
diffusion-path laws under stronger transport and exponential-integrability
assumptions yield sub-critical logarithmic bounds of the form
$(\log n)^{-1/2+\varepsilon}$ for every fixed $\varepsilon>0$; see  Dereich \cite{Dereich2008}. With the path-dependent dyadic
partition and localization according to dyadic ranges of $\Lambda_X$, we reach the Brownian rate under the
present finite-moment assumption, proving
Theorem~\ref{thm:functional-quantization}.

The third ingredient addresses a structural gap between quantization and
empirical sampling. Functional quantization uses a deterministic finite
codebook whose cells generally have unequal masses, whereas an empirical
measure has random support and assigns mass $1/N$ to every observation.
Proposition~\ref{prop:quantization-to-empirical} supplies a general transfer
principle that requires only a quantization bound of order
$(\log n)^{-1/2}$ and a $\rho$-moment. The argument truncates the law, projects
both the truncated law and its empirical measure onto a common finite
codebook, and estimates the resulting multinomial fluctuation of the cell
masses. A bounded-difference estimate controls deviations of the truncated
empirical Wasserstein functional, while the moment assumption controls the
polynomial tail remainder. Balancing these terms retains the scale
$(\log N)^{-1/2}$, and
Corollary~\ref{cor:quantization-to-empirical-mean} gives the mean rate.

\subsection{Related work}\label{subsec:related-work}

Several general theories relate empirical Wasserstein convergence to the
geometry and tails of the underlying measure. Fournier and Guillin
\cite{FournierGuillin2015} establish nonasymptotic moment and concentration
bounds in Euclidean spaces under moment assumptions. Boissard
\cite{Boissard2011} derives nonasymptotic $1$-Wasserstein deviation bounds
on Polish spaces from transport--entropy inequalities and metric-entropy
estimates, with applications to Gaussian measures and diffusion laws.
Bolley \cite{Bolley2010} develops
path-space concentration inequalities and applies them to empirical path
measures of interacting diffusions. Through covering numbers, Boissard and Le Gouic
\cite{BoissardLeGouic2014} bound the expected $p$-Wasserstein distance and
obtain the order $(\log N)^{-1/2}$ for the expected $2$-Wasserstein distance
to Wiener measure under the supremum norm.
Weed and Bach \cite{WeedBach2019} bound asymptotic rates on compact metric
spaces through upper and lower Wasserstein dimensions, while Lei
\cite{Lei2020} obtains matching upper and lower bounds over distribution
classes in separable Hilbert spaces defined by ellipsoidal moment conditions.
These results explain how infinite-dimensional geometry governs the empirical rate.
For the It\^o processes in Assumption~\ref{assum:regularity}, we control the
path-space complexity through a supremum-norm quantization bound of order
$(\log n)^{-1/2}$.

A complementary line of work studies the process-specific complexity through
functional quantization. Dereich, Fehringer, Matoussi, and Scheutzow
\cite{DereichEtAl2003} relate quantization of Gaussian measures on Banach
spaces to their small-ball probabilities. Dereich and Scheutzow
\cite{DereichScheutzow2006} obtain high-resolution formulas for fractional
Brownian motion under the supremum norm, which give the Brownian rate
$(\log n)^{-1/2}$ when the Hurst parameter is $1/2$. For diffusion processes,
Luschgy and Pag\`es \cite{LuschgyPages2006} construct quantizers attaining the
Brownian logarithmic rate for classes of one-dimensional diffusions and
selected multidimensional models through Lamperti-type transformations
of Brownian quantizers. Creutzig, Dereich, M\"uller-Gronbach, and Ritter
\cite{CreutzigEtAl2009} establish the same order under supremum-norm
distortion for autonomous multidimensional diffusions under Lipschitz and
smoothness assumptions and nondegeneracy at the initial point.
Dereich \cite{Dereich2008} obtains supremum-norm high-resolution asymptotics
for diffusions with scalar volatility under H\"older regularity and
linear-growth assumptions. His leading constant is determined by the
quadratic variation of the martingale component.
General progressively measurable It\^o integrands are already treated by
Luschgy and Pag\`es \cite{LuschgyPages2008}, who obtain the Brownian
logarithmic upper bound under uniform-in-time moment bounds of order at
least two, while with path distortion measured in finite $L^s([0,T])$ norms in contrast to the supreme norm considered here.

The rest of the paper is organized as follows. Section~2 states the main
functional quantization, empirical mean, and concentration results.
Section~3 gives the applications to path-dependent
SDEs and McKean--Vlasov particle systems. Section~4 develops the functional
quantization and quantization-to-empirical arguments and proves the main
results.

\section{Main results}
\subsection{Notation and conventions}
\label{subsec:notation}

Throughout the paper, $T>0$ is a fixed time horizon and
$d,m\in\N$ are positive integers. For $k\in\N$, the Euclidean norm on
$\R^k$ is denoted by $\abs{\cdot}$. For a matrix
$A\in\R^{d\times m}$, $\norm{A}$ denotes its Hilbert--Schmidt norm.

We write
    $\Cc^d:=C([0,T];\R^d)$
for the space of continuous $\R^d$-valued paths on $[0,T]$, equipped
with the supremum norm and the associated metric
\[
    \norm{x}_\infty
    :=
    \sup_{0\leq t\leq T}\abs{x_t},
    \qquad 
    d_\infty(x,y):=\norm{x-y}_\infty,
    \qquad x,y\in\Cc^d.
\]
The space $\Cc^d$ is always endowed with the Borel $\sigma$-field
generated by $d_\infty$. For $x\in\Cc^d$ and $t\in[0,T]$, the stopped
path $x_{t\wedge\cdot}\in\Cc^d$ is defined by
    $x_{t\wedge\cdot}(s):=x_{t\wedge s},
     0\leq s\leq T.$
The symbol $0$ denotes either the zero vector or the zero path,
according to context.

Let $(E,d_E)$ be a Polish metric space. We denote by $\Pc(E)$ the set
of Borel probability measures on $E$ and, for $p\geq1$, by
\[
    \Pc_p(E)
    :=
    \left\{
        \mu\in\Pc(E):
        \int_E d_E(x,x_0)^p\,\mu(dx)<\infty
    \right\},
\]
where $x_0\in E$ is arbitrary. This definition does not depend on the
choice of $x_0$. For $\mu,\nu\in\Pc(E)$, the set of their couplings is
denoted by
\[
    \Pi(\mu,\nu)
    :=
    \left\{
        \pi\in\Pc(E\times E):
        \pi(\,\cdot\,\times E)=\mu,\ 
        \pi(E\times\,\cdot\,)=\nu
    \right\}.
\]
For
$\eta,\zeta\in\Pc_p(\Cc^d)$, we use the distinct notation
\[
    \Wc_p(\eta,\zeta)
    :=
    \left(
        \inf_{\pi\in\Pi(\eta,\zeta)}
        \int_{\Cc^d\times\Cc^d}
        \norm{x-y}_\infty^p\,\pi(dx,dy)
    \right)^{1/p}
\]
for the $p$-Wasserstein distance on path space.

If $F:E\to E'$ is measurable and $\mu\in\Pc(E)$, then
$F_\#\mu:=\mu\circ F^{-1}$ denotes the pushforward of $\mu$ under $F$.
For an $E$-valued random variable $Z$, its law is denoted by
    $\Lc(Z):=\P\circ Z^{-1}.$
The Dirac probability measure at $x$ is denoted by $\delta_x$.

Unless stated otherwise, all random variables and stochastic processes
are defined on a filtered probability space
    $(\Omega,\Fc,(\Fc_t)_{0\leq t\leq T},\P)$
satisfying the usual conditions and carrying an $m$-dimensional
$(\Fc_t)$-Brownian motion $W$. For a continuous local martingale $M$,
its quadratic variation is denoted by $\langle M\rangle$. 

\vspace{0.5em}

The following assumption packages the size of the initial condition and of
the two It\^o characteristics into a single random upper bound $\Lambda_X$.
Radially truncating the initial condition and integrands at a deterministic
level $G$ produces a bounded process that agrees with $X$ on
$\{\Lambda_X\leq G\}$. The moment condition on $\Lambda_X$ allows us to
combine the quantization estimates for these truncated processes and to
control the contribution from large values of $\Lambda_X$.
\begin{assumption}
\label{assum:regularity}
Fix $\rho>1$. Let $X_0$ be an $\mathcal F_0$-measurable
$\mathbb R^d$-valued random variable, and let
\[
a:[0,T]\times\Omega\longrightarrow\mathbb R^d,
\qquad
H:[0,T]\times\Omega\longrightarrow\mathbb R^{d\times m}
\]
be progressively measurable processes. We assume that the process $X$
admits the It\^o representation
\[
X_t
=
X_0
+
\int_0^t a_s\,ds
+
\int_0^t H_s\,dW_s,
\qquad
0\leq t\leq T,
\]
almost surely.

Define
\[
\Lambda_X
:=
1
\vee |X_0|
\vee \operatorname*{ess\,sup}_{0\leq s\leq T}|a_s|
\vee \operatorname*{ess\,sup}_{0\leq s\leq T}\|H_s\|,
\]
where the essential suprema are taken with respect to Lebesgue measure
on $[0,T]$. We assume that
\[
\mathbb E\Lambda_X^\rho<\infty.
\]
\end{assumption}
\begin{remark}
\label{remark:scope-general-ito}
$\mathrm{(i)}$
Assumption~\ref{assum:regularity} is a condition on the realized It\^o
characteristics and does not rely on a particular coefficient representation. In
particular, $a$ and $H$ may be path dependent, random, and merely
progressively measurable. We do not require continuity in the state variable, Markov
property, ellipticity, or lower bound on $H$. The assumption also
includes degenerate and deterministic processes.

\vspace{0.5em}
\noindent$\mathrm{(ii)}$ In fact, the It\^o process in Assumption \ref{assum:regularity} can be extended to continuous adapted processes $X$ with $\Lc(X)=\mu$ such that
\begin{equation}\label{eq:semimartingale-rep}
 X_t=X_0+\int_0^t a_s\,\dd s+M_t,
 \qquad 0\le t\le T,
\end{equation}
where $a$ is progressively measurable, $M$ is a $d$-dimensional continuous local martingale with $M_0=0$, and
\[
 \langle M\rangle_t=\int_0^t c_s\,\dd s
\]
for a predictable process $c_s$ with values in the nonnegative symmetric matrices.  Correspondingly,
\begin{equation}\label{eq:envelope}
 \Lambda_X=1\vee\abs{X_0}\vee\esssup_{0\le s\le T}\abs{a_s}
 \vee\esssup_{0\le s\le T}\sqrt{Tr(c_s)}.
\end{equation}
\end{remark}

\subsection{Functional quantization rate}\label{subsec:quantization-rate}
For a $\Cc^d$-valued random variable $Z$, $q\geq1$, and $n\in\N$, define
the $n$-point functional quantization error by
\begin{equation}
    e_{n,q}(Z)
    :=
    \inf_{\substack{\Ac\subset\Cc^d\\1\leq |\Ac|\leq n}}
    \left(
        \E\big[d_\infty(Z,\Ac)^q\big]
    \right)^{1/q},
    \qquad
    d_\infty(z,\Ac)
    :=
    \min_{a\in\Ac}\|z-a\|_\infty.
\end{equation}
The codebook $\Ac$ is deterministic. Our first main theorem gives the
quantization rate under the finite-moment condition in
Assumption~\ref{assum:regularity}. The known Brownian lower bound identifies
the optimal logarithmic exponent over this class.

\begin{theorem}
\label{thm:functional-quantization}
Under Assumption~\ref{assum:regularity}, for every
$1\leq q<\rho$ there exists
$C=C\bigl(d,T,q,\rho,\E\Lambda_X^\rho\bigr)>0$ such that
\begin{equation}
    e_{n,q}(X)\leq C(\log n)^{-1/2},
    \qquad n\geq2.
\end{equation}

The logarithmic rate is optimal over this class. More
precisely, let $B$ be a one-dimensional standard Brownian motion on $[0,T]$
and set $\bar B:=(B,0,\ldots,0)\in\Cc^d$. For every $1\leq q<\rho$, there is
$c_{q,T}>0$ such that
\begin{equation}
    e_{n,q}(\bar B)\geq c_{q,T}(\log n)^{-1/2},
    \qquad n\geq2.
\end{equation}
Since $\Lambda_{\bar B}=1$, the exponent $1/2$ cannot be improved uniformly
over processes satisfying $\E\Lambda_X^\rho\leq M$ for any fixed $M\geq1$.
\end{theorem}

\begin{proof}
See Subsection~\ref{subsec:proofs-main-results}. The lower bound follows
from the Brownian quantization asymptotics of
Dereich and Scheutzow~\cite{DereichScheutzow2006}.
\end{proof}

\begin{remark}
\label{rem:quantization-comparison}
The Brownian exponent $1/2$ is already attained in earlier
supremum-norm quantization results for Brownian motion
\cite{DereichEtAl2003,DereichScheutzow2006} and for diffusion
models satisfying the structural or smoothness assumptions of
\cite{LuschgyPages2006,Dereich2008,CreutzigEtAl2009}. In particular,
Dereich \cite{Dereich2008} obtains supremum-norm high-resolution
asymptotics, including the leading constant, for diffusions with scalar
volatility under his regularity assumptions.
Theorem~\ref{thm:functional-quantization} retains this order for general
progressively measurable It\^o characteristics, with matrix-valued
diffusion integrands, controlled by a common random envelope in $L^\rho$.
Luschgy and Pag\`es
\cite{LuschgyPages2008} indeed treats general It\^o integrands, but measure
path distortion in $L^s([0,T])$ norms. 

\end{remark}

\begin{remark}
\label{remark:uniform-optimality}
\noindent$\mathrm{(i)}$ The quantization lower bound in
Theorem~\ref{thm:functional-quantization} also proves optimality of the
empirical exponent. Let $\bar B$ be the embedded Brownian motion used in
that theorem. For every probability measure $\nu$ supported on at
most $N$ paths,
\begin{equation}\label{eq:lower-bound-BM}
    \Wc_p(\Lc(\bar B),\nu)
    \ge e_{N,p}(\bar B)
    \ge c_{p,T}(\log N)^{-1/2}.
\end{equation}
Applying \eqref{eq:lower-bound-BM} to each realization of the empirical measure gives the matching
lower order. 

\vspace{0.5em}
\noindent$\mathrm{(ii)}$ The condition $\rho>p$ is essential to our proofs. In the
quantization argument it ensures summability of the errors over the dyadic
ranges of $\Lambda_X$. In the empirical transfer argument it makes the
truncation bias negligible. The logarithmic exponent $1/2$ comes from
Proposition~\ref{prop:adaptive-dyadic-coding} and is therefore independent of the
size of the moment gap.

\vspace{0.5em}
\noindent$\mathrm{(iii)}$
The main difficulty lies in Proposition \ref{prop:adaptive-dyadic-coding}, since this error is the dominant rate, details are in Remark \ref{rem:adaptive-coding-literature}. The usual and most common idea in current literatures is to take use of $\alpha$-H\"older continuity of SDE solutions to make time and space discretization with $\alpha < 1/2$. This will result in the rate of $C_\eps(\log N)^{-1/2+\eps}$, and $C_\eps$ will go to $+\infty$ when $\eps$ tends to $0$, or consider nondegenerate or deterministic volatility so that one can transfer the process to a Brownian motion.
\end{remark}

The moment condition in Assumption~\ref{assum:regularity} concerns a
time-uniform random envelope instead of a uniform-in-time bound on
the moments of the integrands. The following example shows that these
conditions cannot, in general, be interchanged. Indeed, we show that, already for $\rho=2$,
the latter condition does not guarantee the $(\log n)^{-1/2}$
quantization rate, even when $\|X\|_\infty$ has moments of every order.

\begin{example}
\label{ex:uniform-moments-counterexample}
Let $T=d=m=1$. Take an $\mathcal F_0$-measurable pair $(K,J)$,
independent of Brownian motion $W$, such that
\[
\mathbb P(K=k)=\frac{1}{k(k+1)},\qquad
\mathbb P(J=j\mid K=k)=2^{-k^6},
\quad k\ge1,\quad 1\le j\le 2^{k^6}.
\]
Set
\[
H_t=2^{K^6/2}
\mathbf 1_{((J-1)2^{-K^6},\,J2^{-K^6}]}(t),
\qquad
X_t=\int_0^t H_s\,dW_s.
\]
Then $H$ is predictable and $\mathbb E H_t^2=1$ for $t>0$, so
\[
\sup_{t\in[0,1]}\mathbb E H_t^2=1,
\qquad
\mathbb E\!\left[\operatorname*{ess\,sup}_{t\in[0,1]}H_t^2\right]
=\sum_{k\ge1}\frac{2^{k^6}}{k(k+1)}=\infty.
\]
Also $\mathbb E\|X\|_\infty^r<\infty$ for every $r>0$.

For fixed $k$, each event $\{K=k,J=j,X_1\ge1\}$ has probability
$\mathbb P(W_1\ge1)/[k(k+1)2^{k^6}]$, since $X_1$ follows $\Nc(0,1)$ conditioning on $\{K=k,J=j\}$. 
If $j<j'$, then on
$\{K=k,J=j,X_1\ge1\}$, $X_{j2^{-k^6}}=X_1\ge1$. On
$\{K=k,J=j',X_1\ge1\}$, $X_{j2^{-k^6}}=0$. Thus every path from the first event is at
supremum distance at least $1$ from every path from the second.

Now choose any $n$ approximating paths. By the triangle inequality,
no approximating path can be within supremum distance $1/2$ of paths
from two different values of $j$. Hence, for at least $2^{k^6}-n$
values of $j$, every path on $\{K=k,J=j,X_1\ge1\}$ is at distance
at least $1/2$ from all $n$ approximating paths. It follows
that, for $n\le2^{k^6}$,
\[
e_{n,1}(X)\ge
\frac{\mathbb P(W_1\ge1)}{2k(k+1)}
\left(1-\frac{n}{2^{k^6}}\right).
\]
Taking $n=2^{k^6-1}$ yields
\[
e_{n,1}(X)\sqrt{\log(n)}
=
e_{2^{k^6-1},1}(X)\sqrt{\log(2^{k^6-1})}
\ge
\frac{\mathbb P(W_1\ge1)\sqrt{(k^6-1)\log 2}}
{4k(k+1)}
\longrightarrow\infty.
\]
Thus the rate $(\log n)^{-1/2}$ fails under the weakened condition with
$\rho=2$, both for quantization and for the empirical Wasserstein
approximation below, since $\mathcal W_1(\mathcal L(X),\nu)\ge e_{N,1}(X)$
for every $N$-point measure $\nu$.
\end{example}

\subsection{Mean Wasserstein rate}\label{subsec:mean-rate}
The next main theorem identifies the sampling rate in path space.
Its significance is that the logarithmic exponent is the Brownian exponent
$1/2$ throughout the full class of Assumption~\ref{assum:regularity}; the
possibly unbounded and random size of the characteristics affects the
constant but does not produce an $\varepsilon$-loss in the exponent.
\begin{theorem}[Empirical path-law rate]
\label{thm:endpoint-path-law}
Let $\rho >p\ge1$. Under Assumption \ref{assum:regularity}, let
\begin{equation}
    \mu:=\Lc(X),
    \qquad
    \mu_N:=\frac1N\sum_{i=1}^N\delta_{X_i},
\end{equation}
where $X_1,\dots,X_N$ are independent copies of $X$. Then there exists
$C=C\bigl(d,T,p,\rho,\E\Lambda_X^\rho\bigr)>0$ such that
\begin{equation}
    \E\big[\Wc_p(\mu_N,\mu)\big]
    \le
    C(\log N)^{-1/2},
    \qquad N\ge2.
\end{equation}
\end{theorem}

\begin{proof}
See Subsection~\ref{subsec:proofs-main-results}.
\end{proof}

\begin{remark}
\label{rem:empirical-comparison}
For Wiener measure under the supremum norm, the endpoint empirical rate
is already known: Boissard and Le Gouic
\cite[Theorem~1.4]{BoissardLeGouic2014} obtain it in $\Wc_2$ from
Gaussian small-ball estimates, and Boissard's Gaussian example
\cite[Example~2.12]{Boissard2011} yields it in $\Wc_1$.
These sharp Gaussian results should be distinguished from estimates
based on subcritical H\"older regularity.

For example, Bolley's independent-sample estimate
\cite[Proposition~3.1]{Bolley2010}, as applied to diffusion paths,
assumes a transport inequality and exponential integrability of an
$\eta$-H\"older norm for some $\alpha<\eta<1/2$, and leads to
an order of $(\log N)^{-\alpha}$. 
The general diffusion estimate of \cite{Boissard2011} has a similar loss, detailed in
Remark~\ref{rem:sde-comparison}.

Theorem~\ref{thm:endpoint-path-law} retains the endpoint when passing
from deterministic quantization to empirical measures, for every
$1\leq p<\rho$ under Assumption~\ref{assum:regularity}.
Lei's sharp functional-space results \cite{Lei2020} concern specified
Hilbert-norm classes with coordinate-decay assumptions and do not
directly provide this supremum-norm conclusion.
\end{remark}

\subsection{Concentration inequality}
\label{sec:concentration}

The mean estimate describes the typical approximation scale but does not
separate moderate fluctuations from rare observations with large supremum
norm. The next theorem makes this separation explicit. The exponential term is produced
by bounded differences after truncation, whereas the polynomial term records
the cost of controlling the unbounded tails using only a finite
$\rho$-moment.

\begin{theorem}
\label{cor:polynomial-path-law-concentration}
Let $\rho>p\geq1$, and let $\mu$ and $\mu_N$ be as in
Theorem~\ref{thm:endpoint-path-law}. Set
\begin{equation}
    \beta_{\rho,p}
    :=
    \min\left\{\frac{\rho}{p}-1,\frac{\rho}{2p}\right\}.
\end{equation}
Under Assumption~\ref{assum:regularity}, there exist constants $c_p>0$ and
$C=C\bigl(d,T,p,\rho,\E\Lambda_X^\rho\bigr)>0$ such that, for every
$N\geq2$ and $t>0$,
\begin{equation}\label{eq:main_concen}
    \P\left(
        \Wc_p(\mu_N,\mu)
        >C(\log N)^{-1/2} + t
    \right)
    \le
    \exp\left(
        -\frac{c_pNt^{2p}}{(\log N)^{p^2/(\rho-p)}}
    \right)
    +CN^{-\beta_{\rho,p}}t^{-\rho}.
\end{equation}
After increasing $C$ if necessary, one also has
\begin{equation}
    \P\left(
        \Wc_p(\mu_N,\mu)
        >C(\log N)^{-1/2}
    \right)
    \le
   CN^{-\beta_{\rho,p}}(\log N)^{\rho/2},
    \qquad N\ge2.
\end{equation}
\end{theorem}

\begin{proof}
See Subsection~\ref{subsec:proofs-main-results}.
\end{proof}

\begin{remark}
\label{remark:two-deviation-mechanisms}
The two terms in \eqref{eq:main_concen} have different origins and should not
be conflated. The exponential term controls fluctuations of the truncated
empirical measure on a bounded path-space ball. The polynomial term controls
the empirical truncation error and is dictated by the available moment. 
\end{remark}

For an exponential deviation around the mean, we record a standard transport-inequality formulation.
It is the path-space counterpart of the approach in Boissard \cite{Boissard2011}
and Bolley \cite{Bolley2010}.

\begin{proposition}
\label{prop:transport-concentration}
Let $\rho>p\geq1$. Under Assumption~\ref{assum:regularity}, let $\mu$ and
$\mu_N$ be as in Theorem~\ref{thm:endpoint-path-law}. Suppose in addition
that, for some $C_T>0$, $\mu$
satisfies the quadratic transport inequality
\begin{equation}
    \Wc_2^2(\nu,\mu)
    \le
    2C_T\,\mathrm H(\nu\mid\mu)
\end{equation}
for every probability measure $\nu$ on $\Cc^d$, where $\mathrm H$ denotes
relative entropy. Then there exists
$C=C\bigl(d,T,p,\rho,\E\Lambda_X^\rho\bigr)>0$ such that, for every $t>0$
and $N\ge2$,
\begin{equation}
    \P\left(
        \Wc_p(\mu_N,\mu)
        >C(\log N)^{-1/2}+t
    \right)
    \le
    \exp\left(
        -\frac{N^{2/\max\{p,2\}}t^2}{2C_T}
    \right).
    \label{eq:transport-concentration}
\end{equation}
In particular, the exponent in \eqref{eq:transport-concentration} is
$-Nt^2/(2C_T)$ when $1\le p\le2$.
\end{proposition}

\begin{proof}
See Subsection~\ref{subsec:proofs-main-results}.
\end{proof}

\subsection{Arbitrarily slow convergence rate without It\^o decomposition}
\label{subsec:arbitrarily-slow-empirical-laws}

The logarithmic estimate above depends essentially on It\^o's semimartingale decomposition.  If one only assumes that the common law is a Borel probability measure on path space, then there is not any meaningful rate of empirical Wasserstein convergence, even among laws with compact support.  
This is in sharp contrast to the finite-dimensional case. Indeed, Fournier and Guillin \cite[Theorem~1]{FournierGuillin2015} imply that, for every
compactly supported law on $\R^d$ and every $p\geq1$, there exist constants
$C,\alpha>0$ such that
$\E[W_p(\mu_N,\mu)]\leq C N^{-\alpha}$.  More generally, Weed and Bach
\cite[Theorem~1]{WeedBach2019} establish a polynomial expected rate on a compact
metric space whenever the upper $p$-Wasserstein dimension of the measure is
finite. Inspired by Berend and Kontorovich
\cite[Proposition~4]{BerendKontorovich2012}, the next proposition constructs a compactly supported law on $\Cc^d$ whose approximation by every
$N$-point measure, and hence whose empirical Wasserstein convergence, is bounded
below by an arbitrarily slowly decaying sequence.

\begin{proposition}
\label{prop:arbitrarily-slow-empirical-laws}
Fix $p\geq 1$, and let $(r_N)_{N\geq 1}$ be a nonincreasing sequence satisfying
$0<r_N\leq 1$ and $r_N\to 0$.  There exists a measure
$\mu\in\mathcal P(\Cc^d)$ whose support is a compact subset of
$\{x\in C^d:\|x\|_\infty\leq 1\}$ such that, for every $N\geq 1$ and every
probability measure $\nu$ supported on at most $N$ paths,
\[
\Wc_p(\mu,\nu)\geq 2^{-1/p}r_N.
\]
In particular, if $X^1,\ldots,X^N$ are independent with common law $\mu$ and
\[
\mu_N=\frac1N\sum_{i=1}^N\delta_{X^i},
\]
then, for every realization of the sample,
\[
\Wc_p(\mu_N,\mu)\geq 2^{-1/p}r_N.
\]
\end{proposition}

\begin{proof}
    See Subsection \ref{sec:slowrate}.
\end{proof}

\section{Applications}
We now illustrate how the abstract results are used. The first application is
a direct verification theorem for a single path-dependent SDE. The second is a quantitative convergence result for a path dependent interacting particle system. In particular, the
conditional i.i.d.\ empirical approximation of the Mckean-Vlasov equation is propagated through
the mean-field dynamics.
\subsection{Convergence of empirical process law of path-dependent SDE}
The coefficient functions below may depend on the complete stopped path and
are assumed only to be Borel measurable. Linear growth is sufficient because
it yields a supremum-moment estimate for the solution and hence an integrable
random bound on the realized drift and diffusion characteristics.
\begin{assumption}
\label{ass:SDE}
Fix $\rho>1$. The maps
\[
b:[0,T]\x \Cc^d\longrightarrow\mathbb{R}^d,
\qquad
\sigma:[0,T]\x \Cc^d
\longrightarrow\mathbb{R}^{d\times m}
\]
are Borel measurable. There exists a constant $L\geq 1$ such that
\[
|b(t,x)|+\|\sigma(t,x)\|
\leq L(1+\|x\|_\infty),
\qquad
(t,x)\in[0,T]\times\Cc^d.
\]
The initial condition $X_0$ is $\mathcal{F}_0$-measurable and satisfies
\[
\mathbb{E}|X_0|^\rho<\infty.
\]
We assume that there exists an $\mathbb{R}^d$-valued continuous
$(\mathcal{F}_t)$-adapted process $X$ satisfying
\[
X_t
=
X_0
+
\int_0^t b(s,X_{s\wedge\cdot})\,ds
+
\int_0^t \sigma(s,X_{s\wedge\cdot})\,dW_s,
\qquad
0\leq t\leq T,
\]
almost surely.
\end{assumption}

The next proposition shows that the abstract endpoint theorem is stable under
this standard linear-growth formulation. No continuity or Lipschitz condition
is needed for the convergence estimate itself once existence of a continuous
adapted solution has been assumed.
\begin{prop}
\label{prop:SDE}
Let $\rho >p\ge1$. Under Assumption \ref{ass:SDE}, let
$X_1,\dots,X_N$ be independent copies of $X$. Then there exists
$C=C\bigl(d,T,p,\rho,L,\E\abs{X_0}^\rho\bigr)>0$ such that
\begin{equation}
    \E\left[
        \Wc_p\left(
            \frac1N\sum_{i=1}^N\delta_{X_i},
            \Lc(X)
        \right)
    \right]
    \le
    C(\log N)^{-1/2},
    \qquad N\ge2.
\end{equation}
\end{prop}

\begin{proof}
See Subsection~\ref{subsec:proof-sde-application}.
\end{proof}

\begin{remark}
\label{rem:sde-comparison}
For a direct diffusion comparison, Boissard
\cite[Example~2.11]{Boissard2011} considers a Markov SDE with
deterministic initial condition and locally Lipschitz coefficients;
the empirical bound assumes bounded drift and diffusion. Inverting
the displayed deviation estimate in that example gives, for each
fixed $0<\alpha<1/2$,
\[
\E\Wc_1(\mu_N,\mu)
\leq C_\alpha
(\log N)^{-\beta_\alpha}
(\log\log N)^{\gamma_\alpha},
\qquad N\ \text{sufficiently large},
\]
where
\[
\beta_\alpha=\frac{2\alpha}{3-2\alpha}<\frac12,
\qquad
\gamma_\alpha=\frac{1-2\alpha}{3-2\alpha}>0.
\]
Consequently, this particular diffusion bound does not recover the
sharp order $(\log N)^{-1/2}$, even though the same paper's Gaussian
example does, as noted in Remark~\ref{rem:empirical-comparison}.

Luschgy and Pag\`es
\cite{LuschgyPages2008} already allow Borel coefficients of linear
growth for existing Markov diffusion solutions, but their functional
quantization estimates measure distortion in $L^p$ norms of a time integral.
\end{remark}

\subsection{Limit theory for path-dependent McKean--Vlasov SDEs with common noise}

In the second application, we consider a path-dependent McKean--Vlasov SDE and its associated interacting particle system. In order to establish finite-particle convergence of this system, our path law convergence result in Section \ref{subsec:mean-rate} is essential, as the marginal empirical law estimate on $\mathbb R^d$ from Fournier and Guillin \cite{FournierGuillin2015} is not sufficient. The proof compares the interacting system with
conditional independent copies of the Mckean-Vlasov limit. The empirical error of those
copies is controlled by the preceding quantization theory, while the
difference between the two systems is controlled by synchronous coupling and
Gr\"onwall stability.

\vspace{0.5em}

Fix a time horizon $T>0$, integers $d,m, l\geq 1$, and an exponent $p\geq 1$.
Let
\[
(b, \sigma,\sigma_0):[0,T]\times\Cc^d\times\Pc_p(\Cc^d)
\longrightarrow \R^d \x \R^{d\times m} \x \R^{d \x l}
\]
be jointly Borel measurable. 
\begin{assumption}\label{ass:coefficients}
There are constants $L>0$ and $r>p$ such that, for every $t\in[0,T]$,
$x,y\in\Cc^d$, and $\nu,\eta\in\Pc_p(\Cc^d)$,
\[
\begin{split}
    &
    \abs{b(t,x,\nu)-b(t,y,\eta)}
    ~+~
    \norm{\sigma(t,x,\nu)-\sigma(t,y,\eta)}
    ~+~
    \norm{\sigma_0(t,x,\nu)-\sigma_0(t,y,\eta)}
    \\\leq&~
    L\bigl(\norm{x-y}_\infty+\Wc_p(\nu,\eta)\bigr),
\end{split}
\]
and
\[
\abs{b(t,\mathbf 0,\delta_{\mathbf 0})}
+
\norm{\sigma(t,\mathbf 0,\delta_{\mathbf 0})}
+
\norm{\sigma_0(t,\mathbf 0,\delta_{\mathbf 0})}
\leq L.
\]
The initial random variable $\xi$ is independent of the driving Brownian motions $W$, $W^0$ and satisfies
\[
\E\abs{\xi}^{r}<\infty.
\]
\end{assumption}

Let $Y$ be the unique strong solution of the path-dependent McKean--Vlasov
stochastic differential equation
\[
\begin{split} 
Y_t
=
\xi
+
\int_0^t b(s,Y_{s\wedge\cdot},\mu_s)\dd s
+
\int_0^t \sigma(s,Y_{s\wedge\cdot},\mu_s)\dd W_s
+
\int_0^t \sigma_0(s,Y_{s\wedge\cdot},\mu_s)\dd W^0_s,
\\
\mu_s:=\Lc(Y_{s\wedge\cdot}|W^0),
\qquad
\mu:= \mu_T.
\end{split}
\]

For each $N\geq2$, let $(\xi^i,W^i)_{0\leq i\leq N}$ be independent copies
of $(\xi,W)$, and let the interacting particles solve
\[
X_t^{i,N}
=
\xi^i
+
\int_0^t b(s,X_{s\wedge\cdot}^{i,N},\mu_s^N)\dd s
+
\int_0^t \sigma(s,X_{s\wedge\cdot}^{i,N},\mu_s^N)\dd W_s^i
+
\int_0^t \sigma_0(s,X_{s\wedge\cdot}^{i,N},\mu_s^N)\dd W_s^0,
\]
where
\[
\mu_s^N
:=
\frac1N\sum_{j=1}^N\delta_{X_{s\wedge\cdot}^{j,N}}
\in\Pc_p(\Cc^d),
\qquad
\mu^N := \mu^N_T.
\]

The Lipschitz and origin bounds imply the linear-growth estimate
\[
\abs{b(t,x,\nu)}+\norm{\sigma(t,x,\nu)}+\norm{\sigma_0(t,x,\nu)}
\leq
L\bigl(1+\norm{x}_\infty+\Wc_p(\nu,\delta_{\mathbf 0})\bigr).
\]
Under these assumptions, the path-dependent McKean--Vlasov equation has a
unique strong solution by the standard fixed-point argument on stopped-path
law flows. The particle system also has a unique strong solution. Indeed, for
two path configurations $(x^1,\ldots,x^N)$ and $(y^1,\ldots,y^N)$ in
$(\Cc^d)^N$, matching paths with the same index gives, for every
$t\in[0,T]$,
\[
\Wc_p\left(
\frac1N\sum_{j=1}^N\delta_{x_{t\wedge\cdot}^j},
\frac1N\sum_{j=1}^N\delta_{y_{t\wedge\cdot}^j}
\right)
\leq
\left(
\frac1N\sum_{j=1}^N
\norm{x_{t\wedge\cdot}^j-y_{t\wedge\cdot}^j}_\infty^p
\right)^{1/p}.
\]
Consequently, the particle coefficients are globally Lipschitz in the
stopped path configuration, and the usual Picard iteration for functional
SDEs applies.

\begin{prop}\label{thm:app}
Under Assumption~\ref{ass:coefficients}, there is a constant
\[
C=C\bigl(p,r,d,m,T,L,\E\abs{\xi}^{r}\bigr)
\]
such that, for every $N\geq2$,
\[
\E\Wc_p(\mu^N,\mu)
\leq
C(\log N)^{-1/2}.
\]
\end{prop}

\begin{proof}
See Subsection~\ref{subsec:proof-mckean-vlasov-application}.
\end{proof}

\begin{remark}
\label{rem:common-noise-comparison}
For interacting diffusions, Bolley \cite[Theorem~1.1]{Bolley2010}
treats globally Lipschitz Markovian drift and convolution-type
interaction, constant
diffusion, and no common noise, with a square-exponential initial
moment. His path-space concentration estimate has the subcritical
sample-size dependence discussed in
Remark~\ref{rem:empirical-comparison}.

Bayraktar, He, and Kim \cite[Theorem~3.10]{Erhan2025} prove convergence
of empirical full path measures in expected squared $\Wc_2$ for
heterogeneous graphon systems with common noise, under bounded
Lipschitz coefficients, uniform initial moments of order above two,
and regularity assumptions on the graphon and initial-law profile.
They do not give a quantitative rate for this
path-law convergence, as explicitly noted in their Remark~3.12(ii).
\end{remark}

\section{Technical results and proofs}
\subsection{Proofs of the main results}
\label{subsec:proofs-main-results}
The argument in this subsection is organized in three layers. First,
Proposition~\ref{prop:adaptive-dyadic-coding} constructs a codebook by a
pathwise dyadic partition, from which
Lemma~\ref{lem:bounded-functional-quantization} derives the quantization
bound for bounded characteristics. Second,
Proposition~\ref{prop:endpoint-functional-quantization} localizes according
to dyadic levels of $\Lambda_X$ and proves
the upper bound in Theorem~\ref{thm:functional-quantization}. Third, we transfer
the quantization estimate to equal-weight empirical measures by truncation
and finite-codebook sampling.

\subsubsection{Functional quantization with bounded initial value and characteristics}

The local martingale estimate below controls the conditional probability that
a dyadic interval is split.
\begin{lemma}
\label{lemma:conditional-exponential-supermartingale}
Let $N$ be a real-valued continuous local martingale. Fix deterministic
times $0\le u<v\le T$, and suppose that, for some deterministic $V>0$,
\begin{equation}
    \langle N\rangle_v-\langle N\rangle_u\le V
    \qquad\text{almost surely}.
\end{equation}
Then, for every $x>0$,
\begin{equation}
    \P\left(
        \sup_{u\le t\le v}|N_t-N_u|\ge x
        \,\middle|\,\Fc_u
    \right)
    \le
    2\exp\left(-\frac{x^2}{2V}\right)
    \qquad\text{almost surely}.
\end{equation}
\end{lemma}

\begin{proof}
For $\lambda>0$, define, for $u\le t\le v$,
\begin{equation}
    Z_t^\lambda
    :=
    \exp\left(
        \lambda(N_t-N_u)
        -\frac{\lambda^2}{2}
        (\langle N\rangle_t-\langle N\rangle_u)
    \right).
\end{equation}
This is a positive local martingale starting from one at time $u$, 
hence it is a supermartingale. Let
\begin{equation}
    \tau_+
    :=
    \inf\{t\in[u,v]:N_t-N_u\ge x\}\wedge v.
\end{equation}
Continuity of $N$ and the bracket bound imply that, on
$\{\sup_{u\le t\le v}(N_t-N_u)\ge x\}$,
\begin{equation}
    Z_{\tau_+}^\lambda
    \ge
    \exp\left(\lambda x-\frac{\lambda^2V}{2}\right).
\end{equation}
Conditional optional sampling therefore gives
\begin{align}
    1
    \ge
    \E\left[Z_{\tau_+}^\lambda\mid\Fc_u\right] \
    \ge
    \exp\left(\lambda x-\frac{\lambda^2V}{2}\right)
    \P\left(
        \left\{\sup_{u\le t\le v}(N_t-N_u)\ge x\right\}
        \,\middle|\,\Fc_u
    \right).
\end{align}
Taking $\lambda=x/V$ yields
\begin{equation}
    \P\left(
        \left\{\sup_{u\le t\le v}(N_t-N_u)\ge x\right\}
        \,\middle|\,\Fc_u
    \right)
    \le
    \exp\left(-\frac{x^2}{2V}\right).
\end{equation}
Applying the same argument to $-N$ and taking a union bound proves the
claim.
\end{proof}

\begin{proposition}
\label{prop:adaptive-dyadic-coding}
Let $Y$ be an $\R^d$-valued continuous semimartingale
of the form
\begin{equation}
    Y_t
    =
    Y_0
    +\int_0^t a_s\,ds
    +\int_0^t H_s\,dW_s,
    \qquad 0\le t\le T,
\end{equation}
where $a$ and $H$ are progressively measurable and, for some $G\ge1$,
\begin{equation}
    |Y_0|\le G,
    \qquad
    |a_s|\le G,
    \qquad
    \|H_s\|\le G
\end{equation}
almost surely, with the last two bounds holding for almost every $s$.
There exist positive constants $c_{d,T}$ and $C_{d,T}$, depending only on $d$
and $T$, such
that for every $0<\delta\le G$ there are a deterministic finite set
$\Ac_\delta\subset\Cc^d$, containing the zero path, and an event
$\Gc_\delta$ satisfying
\begin{equation}
    \log|\Ac_\delta|
    \le
    C_{d,T}\frac{G^2}{\delta^2},
    \qquad
    \P(\Gc_\delta^c)
    \le
    \exp\left(-c_{d,T}\frac{G^2}{\delta^2}\right),
\end{equation}
and
\begin{equation}
    d_\infty(Y,\Ac_\delta)
    \le C_{d,T}\delta
    \qquad\text{on }\Gc_\delta.
\end{equation}
\end{proposition}
\begin{remark}
\label{rem:adaptive-coding-literature}
The principal probabilistic difficulty is that the partition is random and
path-dependent. Moreover, the split indicators are not independent, since
the volatility \(H\) is allowed to be an arbitrary bounded progressively
measurable process, and splits must be controlled simultaneously over
infinitely many dyadic levels. The proof resolves the dependence within
each level by ordering the intervals chronologically and iterating
conditional exponential-moment bounds. Dependence between different
levels is then handled by a generalized H\"older argument with suitably
chosen exponents. This produces an exponential moment bound for the total
number of possible splits and, consequently, an exponential tail for the
size of the realized forest.
\end{remark}

\noindent\textbf{Proof strategy}
The main idea is an adaptive dyadic discretization of time. One starts with
\(n_0\asymp_{d,T}G^2/\delta^2\) root intervals and recursively bisects an
interval \(I=[u,v]\) whenever
\[
    \sup_{t\in I}|Y_t-Y_u|>\delta.
\]
The initial mesh is chosen so that the drift contributes at most a fixed
fraction of \(\delta\) on every possible dyadic interval. Hence a split can
occur only if the martingale part makes an excursion of order \(\delta\).
If an interval at depth \(\ell\) has length \(h_\ell\), boundedness of \(H\)
gives quadratic variation of order at most \(G^2h_\ell\), and a conditional
martingale maximal inequality yields
\[
    \P\!\left(
        I\text{ is split}\mid\Fc_u
    \right)
    \leq
    2d\exp\!\left(
        -c_d\frac{\delta^2}{G^2h_\ell}
    \right)
    \leq
    \exp(-2^\ell).
\]
This is the probabilistic reason why the construction works. The number of
potential intervals at depth \(\ell\) grows like \(n_02^\ell\), whereas the
probability of a \(\delta\)-oscillation decays like
\(\exp(-2^\ell)\). The latter decay dominates the combinatorial growth,
so only \(O_{d,T}(G^2/\delta^2)\) intervals are needed with exponentially
high probability. This is also the origin of the Brownian scaling
\(\delta^{-2}\), and hence of the exponent \(1/2\) in the final logarithmic
quantization rate.
\begin{proof}
Write
\begin{equation}
    M_t:=\int_0^tH_s\,dW_s
\end{equation}
and fix $0<\delta\le G$. The subscripts on constants $C$ or $c$ list all
parameters on which they may depend. Their values may change from line to
line, and they are independent of $G$, $\delta$, and $n_0$.

\smallskip

\noindent {\bf Step 0. Binary-tree convention.}
Let $n_0\in\N$ be a positive integer, to be chosen in Step~4.
Partition $[0,T]$ into $n_0$ intervals of equal length, ordered from left to
right and viewed as the roots of dyadic binary trees. If $I=[u,v]$ is a
node, its children are
\begin{equation}
    I^0:=\left[u,\frac{u+v}{2}\right],
    \qquad
    I^1:=\left[\frac{u+v}{2},v\right].
\end{equation}
Let $\mathcal D_\ell$ be the collection of all potential nodes at depth
$\ell$ below the roots. Thus every $I\in\mathcal D_\ell$ has length
$h_\ell$, where
\begin{equation}
    |\mathcal D_\ell|=n_0 2^\ell,
    \qquad
    h_\ell:=\frac{T}{n_0 2^\ell}.
\end{equation}
A finite dyadic forest $\mathsf F$ is an ordered collection of $n_0$ finite
full binary subtrees rooted at the initial intervals. Write $s(\mathsf F)$
and $L(\mathsf F)$ for its numbers of internal nodes (the intervals that
are split) and leaves (the intervals that are not split). Then
$L(\mathsf F)=n_0+s(\mathsf F)$, and $\mathsf F$ has
$n_0+2s(\mathsf F)$ nodes in total.
Whenever $\mathsf F$ is fixed below, abbreviate $L=L(\mathsf F)$ and
enumerate its leaves from left to right as
\begin{equation}
    [t_0,t_1],\dots,[t_{L-1},t_L],
    \qquad
    0=t_0<\cdots<t_L=T.\label{eq:adaptive-leaf-partition}
\end{equation}
These endpoints are deterministic once $\mathsf F$ is fixed.
With the cutoff $s_\delta$ to be specified in Step 1, let
$
    \mathfrak F_\delta
    :=
    \{\mathsf F:s(\mathsf F)\le s_\delta\}$ denote the set of all forests with at most $s_\delta$ internal nodes. Figure~\ref{fig:dyadic-forest} below illustrates these conventions in the case
$n_0=2$.

\begin{figure}[htbp]
\centering
\begin{tikzpicture}[
    x=1cm,
    y=1cm,
    every node/.style={font=\small,inner sep=1pt}
]
    \draw (-5,2.8) -- (5,2.8);
    \foreach \x in {-5,-2.5,0,1.25,2.5,5}
        \draw (\x,2.7) -- (\x,2.9);
    \node at (-5,2.5) {$0$};
    \node at (-2.5,2.5) {$T/4$};
    \node at (0,2.5) {$T/2$};
    \node at (1.25,2.5) {$5T/8$};
    \node at (2.5,2.5) {$3T/4$};
    \node at (5,2.5) {$T$};

    \node[font=\scriptsize] at (-3.75,2.15) {$I_1^0$};
    \node[font=\scriptsize] at (-1.25,2.15) {$I_1^1$};
    \node[font=\scriptsize] at (0.625,2.15) {$(I_2^0)^0$};
    \node[font=\scriptsize] at (1.875,2.15) {$(I_2^0)^1$};
    \node[font=\scriptsize] at (3.75,2.15) {$I_2^1$};

    \draw (-5,3.4) -- (-5,3.52) -- (0,3.52) -- (0,3.4);
    \draw (0,3.4) -- (0,3.52) -- (5,3.52) -- (5,3.4);
    \node at (-2.5,3.72) {$I_1$};
    \node at (2.5,3.72) {$I_2$};
    \draw (0,3.0) -- (0,3.12) -- (2.5,3.12) -- (2.5,3.0);
    \node[font=\scriptsize] at (1.25,3.25) {$I_2^0$};

    \draw[->] (0,1.9) -- (0,1.35);
    \node[right,font=\scriptsize] at (0.12,1.62) {corresponds to};
    \node at (0,1.05) {binary forest $\mathsf F$};

    \node (i1) at (-3.4,0.45) {$I_1$};
    \node (i10) at (-4.3,-0.4) {$I_1^0$};
    \node (i11) at (-2.5,-0.4) {$I_1^1$};
    \draw (i1) -- (i10);
    \draw (i1) -- (i11);

    \node (i2) at (2.2,0.45) {$I_2$};
    \node (i20) at (1.2,-0.4) {$I_2^0$};
    \node (i21) at (3.2,-0.4) {$I_2^1$};
    \node (i200) at (0.55,-1.3) {$(I_2^0)^0$};
    \node (i201) at (1.85,-1.3) {$(I_2^0)^1$};
    \draw (i2) -- (i20);
    \draw (i2) -- (i21);
    \draw (i20) -- (i200);
    \draw (i20) -- (i201);
\end{tikzpicture}
\caption{An $n_0=2$ example of the correspondence between a dyadic
interval partition and an ordered binary forest.}
\label{fig:dyadic-forest}
\end{figure}
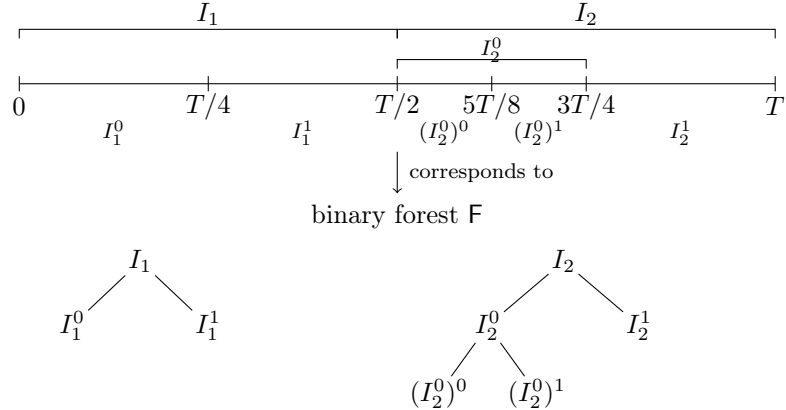

\smallskip

\noindent {\bf Step 1. Count the internal nodes.}
Given a path of $Y$, we define a forest $\mathsf F$ as follows. Starting from the root intervals, recursively
split a dyadic interval $I=[u,v]$ whenever
\begin{equation}
    \sup_{t\in I}|Y_t-Y_u|>\delta.\label{eq:adaptive-splitting-criterion}
\end{equation}
The interval is retained as a leaf when the reverse inequality holds.
After fixing a continuous version of $Y$, uniform continuity on $[0,T]$
shows that the resulting random forest $\mathsf F$ has finitely many nodes
almost surely. With the requirement
\begin{equation}\label{req:n0-1}
    n_0\ge\frac{4GT}{\delta},
\end{equation}
we have
\begin{equation}
    Gh_\ell
    =\frac{GT}{n_0 2^\ell}
    \le\frac{\delta}{4}.
\end{equation}
Noting that
\begin{equation}
    |Y_t-Y_s|
    \le G|t-s|+|M_t-M_s|,
    \qquad 0\le s\le t\le T,
\end{equation}
for any $I\in\mathcal D_\ell$ we have
\begin{equation}
    \left\{\sup_{t\in I}|Y_t-Y_u|>\delta\right\}
    \subset
    \left\{\sup_{t\in I}|M_t-M_u|>\frac{3\delta}{4}\right\}.
    \label{split:event:inclusion}
\end{equation}
For every coordinate $i$ of $M$ and every $t\in I$,
\begin{equation}
    \langle M^i\rangle_t
    -\langle M^i\rangle_u
    \le G^2(t-u)
    \le G^2h_\ell.
\end{equation}
Applying Lemma~\ref{lemma:conditional-exponential-supermartingale} with
$V=G^2h_\ell$ yields
\begin{equation}
    \P\left(
        \sup_{t\in I}|M_t^i-M_u^i|>x
        \,\middle|\,\Fc_u
    \right)
    \le
    2\exp\left(-\frac{x^2}{2G^2h_\ell}\right)
\end{equation}
for every coordinate $i$. Applying \eqref{split:event:inclusion},
a union bound over the $d$ coordinates, and the choice
$x=3\delta/(4\sqrt d)$ gives, for every $I\in\mathcal D_\ell$,
\begin{align}
    \P\left(
        \sup_{t\in I}|Y_t-Y_u|>\delta
        \,\middle|\,\Fc_u
    \right)
    \le
    2d\exp\left(-\frac{9\delta^2}{32dG^2h_\ell}\right)
    =
    2d\exp\left(-\frac{9\delta^2n_0}{32dG^2T}2^\ell\right).\label{split-prob-upperbound}
\end{align}
With the additional requirement
\begin{equation}\label{req:n0-2}
    \frac{9\delta^2n_0}{32dG^2T}
    \ge\log(2d)+1,
\end{equation}
equation \eqref{split-prob-upperbound} gives
\begin{align}
    \P\left(
        \sup_{t\in I}|Y_t-Y_u|>\delta
        \,\middle|\,\Fc_u
    \right)
    \le
    2d\exp\left(-\bigl(\log(2d)+1\bigr)2^\ell\right)
    \le\exp(-2^\ell).
    \label{eq:adaptive-split-probability}
\end{align}
Let $N_\ell$ count the potential intervals at depth $\ell$
that satisfy the splitting criterion, namely,
\begin{equation}
    N_\ell
    :=
    \sum_{I=[u,v]\in\mathcal D_\ell}
    \mathbf 1_{\{\sup_{t\in I}|Y_t-Y_u|>\delta\}}.
\end{equation}
Order the intervals in $\mathcal D_\ell$ chronologically. The event
associated with each preceding interval is measurable at the left endpoint
of the next interval. Iterating \eqref{eq:adaptive-split-probability} therefore
gives, for every $\lambda>0$,
\begin{equation}
    \E\exp(\lambda N_\ell)
    \le
    \left[
        1+\exp(-2^\ell)(e^\lambda-1)
    \right]^{n_0 2^\ell}.
\end{equation}
Choosing the indices $r_l=\pi^2(l+1)^2/6$ and using
the generalized H\"older's inequality gives, for a finite $J$\footnote{Note that for finite $J$, the reciprocal exponents have sum smaller
than one. We thus add the
constant random variable $1$ with reciprocal exponent
    $1-\sum_{\ell=0}^J\frac1{r_\ell}$
to complete the family of H\"older exponents.},
\begin{align}
    \log\E\exp\left(
        \sum_{\ell=0}^J N_\ell
    \right)
    \le
    \sum_{\ell=0}^J
    \frac1{r_\ell}
    \log\E\exp(r_\ell N_\ell) \le
    n_0\sum_{\ell=0}^J
    \frac{2^\ell}{r_\ell}
    \exp(-2^\ell)\big(e^{r_\ell}-1\big)
    \le K_dn_0,\label{eq:adaptive-finite-level-mgf}
\end{align}
for some constant $K_d$ depending only on $d$. Monotone convergence now gives
\begin{equation}
    \E\exp\left(
        \sum_{\ell=0}^\infty N_\ell
    \right)
    \le e^{K_dn_0}.\label{eq:exp:sum-Nl}
\end{equation}
Every internal node of $\mathsf F$ at depth $\ell$ is counted by $N_\ell$,
and hence
\begin{equation}
    s(\mathsf F)\le\sum_{\ell=0}^\infty N_\ell.\label{eq:sF:sum-Nl}
\end{equation}
Set $s_\delta:=\lfloor(K_d+1)n_0\rfloor$. Since $s(\mathsf F)$ is integer-valued, the exponential Markov inequality, in conjunction with \eqref{eq:exp:sum-Nl} and \eqref{eq:sF:sum-Nl}, gives
\begin{equation}
    \P\big(s(\mathsf F)>s_\delta\big)
    =
    \P\big(s(\mathsf F)>(K_d+1)n_0\big)
    \le
    e^{-(K_d+1)n_0}\E e^{s(\mathsf F)}
    \le e^{-n_0}.
\end{equation}
Defining
$\Gc_\delta:=\{s(\mathsf F)\le s_\delta\}$,
we have $\mathsf F\in\mathfrak F_\delta$ on $\Gc_\delta$.

\smallskip
\noindent {\bf Step 2. Construct the codebook.}
Define the coordinatewise lattice quantizer $Q_\delta:\R^d\to\delta\mathbb Z^d$
by
\begin{equation}
    [Q_\delta(x)]^i
    :=
    \delta\left\lceil\frac{x^i}{\delta}-\frac12\right\rceil,
    \qquad 1\le i\le d.
\end{equation}
Thus ties are resolved toward the smaller lattice point. This quantizer
satisfies
\begin{equation}
    |Q_\delta(x)-x|\le\frac{\sqrt d}{2}\delta,
    \qquad
    |x^i-y^i|\le\delta
    \quad\Longrightarrow\quad
    [Q_\delta(x)]^i-[Q_\delta(y)]^i\in\delta\{-1,0,1\}
    \label{quantizer:adjacency}
\end{equation}
for every $1\le i\le d$.
Let
\begin{equation}
    \Lambda_{\delta,G}
    :={}
    \delta\mathbb Z^d
    \cap
    \left[-G-\frac\delta2,G+\frac\delta2\right]^d.
\end{equation}
For $\mathsf F\in\mathfrak F_\delta$, with $L=L(\mathsf F)$ as above,
define the endpoint arrays by
\begin{equation}
    \mathcal Z_{\mathsf F}
    :=
    \left\{
        (z_0,\dots,z_L)\in(\delta\mathbb Z^d)^{L+1}:
        z_0\in\Lambda_{\delta,G},\ 
        z_j-z_{j-1}\in\delta\{-1,0,1\}^d
        \text{ for }1\le j\le L
    \right\}.
\end{equation}
For $\boldsymbol z=(z_0,\dots,z_L)\in\mathcal Z_{\mathsf F}$, let
$\mathsf I_{\mathsf F}[\boldsymbol z]\in\Cc^d$ be the path obtained by linear
interpolation at the leaf endpoints. More precisely, for
$t\in[t_{j-1},t_j]$, set
\begin{equation}
    \mathsf I_{\mathsf F}[\boldsymbol z](t)
    :=
    \frac{t_j-t}{t_j-t_{j-1}}z_{j-1}
    +
    \frac{t-t_{j-1}}{t_j-t_{j-1}}z_j.
\end{equation}
Define
\begin{align}
    \Ac_{\mathsf F}
    :={}
    \left\{
        \mathsf I_{\mathsf F}[\boldsymbol z]:
        \boldsymbol z\in\mathcal Z_{\mathsf F}
    \right\}, 
    \qquad
    \Ac_\delta
    :={}
    \bigcup_{\mathsf F\in\mathfrak F_\delta}\Ac_{\mathsf F}.
\end{align}
In particular, the zero path belongs to every $\Ac_{\mathsf F}$ and hence
to $\Ac_\delta$. The elementary count for ordered full binary forests gives
\begin{equation}
    |\mathfrak F_\delta|
    =
    \sum_{s=0}^{s_\delta}\frac{n_0}{n_0+2s}\binom{n_0+2s}{s}
    \le
    \sum_{s=0}^{s_\delta}2^{n_0+2s}
    \le
    2^{n_0+2s_\delta+1}.
\end{equation}
Since $G/\delta\ge1$, we also have
\begin{equation}
    |\Lambda_{\delta,G}|
    \le
    \left(4\frac{G}{\delta}\right)^d.
\end{equation}
For every $\mathsf F\in\mathfrak F_\delta$, the identity
$L(\mathsf F)=n_0+s(\mathsf F)$ gives
$L(\mathsf F)\le n_0+s_\delta$. Counting the initial lattice point and the
successive leaf increments therefore yields
\begin{align}
    |\Ac_\delta|
    &\le
    |\mathfrak F_\delta|
    \left(4\frac{G}{\delta}\right)^d
    3^{d(n_0+s_\delta)}  
    \le
    2^{n_0+2s_\delta+1}
    \left(4\frac{G}{\delta}\right)^d
    3^{d(n_0+s_\delta)},  \\
    \log|\Ac_\delta|
    &\le
    (n_0+2s_\delta+1)\log2
    +d\log\left(4\frac{G}{\delta}\right)
    +d(n_0+s_\delta)\log3 
    \le C_d n_0+d\log\left(4\frac{G}{\delta}\right).
\end{align}
Here we use $s_\delta\le(K_d+1)n_0$ and $n_0\ge1$.

\smallskip
\noindent {\bf Step 3. Pathwise reconstruction.}
Fix $\omega\in\Gc_\delta$ and set $y:=Y(\omega)$. Denote the realized
forest simply by $\mathsf F\in\mathfrak F_\delta$. With the leaf partition
in \eqref{eq:adaptive-leaf-partition}, set
\begin{equation}
    z_j:=Q_\delta(y_{t_j}),
    \qquad 0\le j\le L.
\end{equation}
For every leaf $[t_{j-1},t_j]$, the reverse inequality in
\eqref{eq:adaptive-splitting-criterion} gives
\begin{equation}
    \sup_{t\in[t_{j-1},t_j]}|y_t-y_{t_{j-1}}|\le\delta.
    \label{leaf:oscillation}
\end{equation}
The bound $|y_0|\le G$ and the adjacency property
\eqref{quantizer:adjacency} therefore imply
\begin{equation}
    z_0\in\Lambda_{\delta,G},
    \qquad
    z_j-z_{j-1}\in\delta\{-1,0,1\}^d,
    \quad 1\le j\le L.
\end{equation}
Thus $(z_0,\dots,z_L)$ belongs to
$\mathcal Z_{\mathsf F}$, and the corresponding codeword
$
    \widehat y
    :=
    \mathsf I_{\mathsf F}[(z_0,\dots,z_L)]$
belongs to $\Ac_{\mathsf F}\subset\Ac_\delta$. For $t\in[t_{j-1},t_j]$, write
\begin{equation}
    \theta:=\frac{t-t_{j-1}}{t_j-t_{j-1}}.
\end{equation}
The leaf bound \eqref{leaf:oscillation} and the quantization-error bound in
\eqref{quantizer:adjacency} give
\begin{align}
    |y_t-\widehat y_t|
    \le
    \left|
        y_t-(1-\theta)y_{t_{j-1}}-\theta y_{t_j}
    \right|  
    +(1-\theta)|y_{t_{j-1}}-z_{j-1}|
    +\theta|y_{t_j}-z_j|  
    \le
    \left(2+\frac{\sqrt d}{2}\right)\delta.
\end{align}
Consequently,
\begin{equation}
    d_\infty(Y(\omega),\Ac_\delta)
    \le
    \|y-\widehat y\|_\infty
    \le
    \left(2+\frac{\sqrt d}{2}\right)\delta
    \qquad\text{for }\omega\in\Gc_\delta.
\end{equation}

\smallskip
\noindent {\bf Step 4. Optimize the choice of $n_0$.}
The preceding argument gives
\begin{equation}
    \P(\Gc_\delta^c)\le e^{-n_0},
    \qquad
    \log|\Ac_\delta|
    \le C_d n_0+d\log\left(4\frac{G}{\delta}\right),
\end{equation}
provided that \eqref{req:n0-1} and \eqref{req:n0-2} hold. Thus it is sufficient that
\begin{equation}
    n_0\ge
    \max\left\{
        \frac{4GT}{\delta},\,
        \frac{32d}{9}\bigl(\log(2d)+1\bigr)
        \frac{TG^2}{\delta^2}
    \right\}.
\end{equation}
The upper bound for the codebook size increases with $n_0$, so we choose
$n_0$ at the smallest order allowed by these constraints with
\begin{equation}
    n_0:=
    \left\lceil
        \bigg(4+\frac{32d}{9}(\log(2d)+1)\bigg)(1+T)\frac{G^2}{\delta^2}
    \right\rceil,
\end{equation}
which satisfies all the requirements above.
Since $G/\delta\ge1$, the logarithmic term in the codebook bound is
absorbed into $C_dn_0$. Hence, for some constants $c_{d,T}, C_{d,T}$ depending only on $d,T$,
\begin{equation}
    \log|\Ac_\delta|
    \le C_dn_0
    \le C_{d,T}\frac{G^2}{\delta^2},
    \qquad
    \P(\Gc_\delta^c)
    \le e^{-n_0}
    \le\exp\left(-c_{d,T}\frac{G^2}{\delta^2}\right).
\end{equation}
Together with the pathwise estimate in Step~3 and the fact that
$0\in\Ac_\delta$, these bounds prove the proposition.
\end{proof}

The codebook estimate yields the following quantization bound for bounded
characteristics, which we use in the localization argument.
\begin{lemma}
\label{lem:bounded-functional-quantization}
Under the hypotheses of Proposition~\ref{prop:adaptive-dyadic-coding}, for
every $q\geq1$ there exists $C_{d,q,T}>0$, depending only on $d,q,T$, such
that
\begin{equation}
    e_{n,q}(Y)\leq C_{d,q,T}G(\log n)^{-1/2},
    \qquad n\geq2.
\end{equation}
\end{lemma}

\begin{proof}
The Burkholder--Davis--Gundy inequality implies
\begin{equation}
    \left(\E\|Y\|_\infty^{2q}\right)^{1/(2q)}
    \le
    G+GT+C_qG\sqrt T
    \le C_{q,T}G.\label{eq:adaptive-Y-moment}
\end{equation}
Let $\Ac_\delta$ and $\Gc_\delta$ be given by
Proposition~\ref{prop:adaptive-dyadic-coding}. Since the zero path belongs
to $\Ac_\delta$, H\"older's inequality and the bounds in that proposition
yield
\begin{align}
    \left(
        \E d_\infty(Y,\Ac_\delta)^q
    \right)^{1/q}
    \le&
    C_{d,T}\delta
    +
    \left(\E\|Y\|_\infty^{2q}\right)^{1/(2q)}
    \P(\Gc_\delta^c)^{1/(2q)}\\ 
    \le&
    C_{d,T}\delta
    +
    C_{q,T}G
    \exp\left(-\frac{c_{d,T}G^2}{2q\delta^2}\right)
    \le C_{d,q,T}\delta.\label{eq:adaptive-distortion}
\end{align}
The final inequality follows from $G/\delta\geq1$. Let
$K_{d,T}:=\max\{1,C_{d,T}\}$, where $C_{d,T}$ is the codebook cardinality
constant in Proposition~\ref{prop:adaptive-dyadic-coding}. For
$\log n\geq K_{d,T}$, take
\begin{equation}
    \delta
    =
    \sqrt{K_{d,T}}G(\log n)^{-1/2}.
\end{equation}
Then $\delta\leq G$ and $\log|\Ac_\delta|\leq\log n$. Substituting this
choice of $\delta$ into \eqref{eq:adaptive-distortion} gives
\begin{equation}
    e_{n,q}(Y)
    \le
    C_{d,q,T}G(\log n)^{-1/2}
\end{equation}
for all such $n$. The finitely many remaining values of $n$ are absorbed into
$C_{d,q,T}$ using \eqref{eq:adaptive-Y-moment}.
\end{proof}

\subsubsection{Localization for integrable characteristics}

Lemma~\ref{lem:bounded-functional-quantization} cannot be applied directly under
Assumption~\ref{assum:regularity}, because the initial value and integrands are
bounded only by the integrable random variable $\Lambda_X$.
The next proposition localizes according to dyadic levels of $\Lambda_X$ and
allocates codebooks of decreasing cardinality to increasingly rare events. The
moment condition makes these errors summable without changing the
rate.
\begin{proposition}
\label{prop:endpoint-functional-quantization}
Under Assumption~\ref{assum:regularity}, for every $q$ satisfying
$1\leq q<\rho$, there exists a constant
$C=C\bigl(d,T,q,\rho,\E\Lambda_X^\rho\bigr)>0$ such that
\[
e_{n,q}(X)
\leq
C(\log n)^{-1/2},
\qquad
n\geq2.
\]
\end{proposition}

\begin{proof}
Fix $q\in[1,\rho)$. Choose $\theta>1$ sufficiently large that, with
$\beta:=1-\frac1\theta$,
one has $\rho\beta>q$.
Such a choice is possible because $q<\rho$. For instance, it is enough
to take
$\theta>\frac{\rho}{\rho-q}$.

\medskip
\noindent\textbf{Step 1. Supremum moment estimate.}
We first record the moment estimate that will also be used to control the
contribution from large values of $\Lambda_X$. The It\^o representation of
$X$ gives
\[
\|X\|_\infty
\leq
|X_0|
+
\int_0^T|a_s|\,ds
+
\sup_{0\leq t\leq T}\left|\int_0^tH_s\,dW_s\right|.
\]
Since one has that
\[
|X_0|\leq\Lambda_X
\qquad\text{and}\qquad
\int_0^T|a_s|\,ds\leq T\Lambda_X,
\]
the Burkholder--Davis--Gundy inequality yields
\begin{equation}
\mathbb E\|X\|_\infty^\rho
\leq
C_\rho\mathbb E|X_0|^\rho
+
C_\rho\mathbb E
\left(\int_0^T|a_s|\,ds\right)^\rho
+
C_\rho\mathbb E
\left(\int_0^T\|H_s\|^2\,ds\right)^{\rho/2} 
\leq
C_{\rho,T}
\mathbb E\Lambda_X^\rho < \infty.\label{eq:endpoint:supremum-moment}
\end{equation}

\medskip
\noindent\textbf{Step 2. Truncation.}
It remains to prove the quantization estimate. It is enough initially to
consider sufficiently large $n$. Set
\[
R_j:=2^{j+1},
\quad j\geq0,
\]
and decompose the probability space into the disjoint events
\[
B_0:=\{\Lambda_X\leq2\},
\qquad
B_j:=\{2^j<\Lambda_X\leq2^{j+1}\},
\quad j\geq1.
\]
Since $\Lambda_X$ is finite almost surely, these events form a partition
of $\Omega$ up to a null set. For $R>0$, let $\pi_R$ denote radial projection onto the closed ball of
radius $R$. We use the same notation for vectors and matrices, with the
Euclidean norm for vectors and the Hilbert--Schmidt norm for matrices.
Thus,
\[
\pi_R(z)
=
\begin{cases}
z, & \text{if }\lvert z\rvert\leq R,\\[1mm]
Rz/\lvert z\rvert, & \text{if }\lvert z\rvert>R,
\end{cases}
\]
with the corresponding definition for matrices. For every $j\geq0$, define the continuous It\^o process
\[
Y_t^{(j)}
=
\pi_{R_j}(X_0)
+
\int_0^t\pi_{R_j}(a_s)\,ds
+
\int_0^t\pi_{R_j}(H_s)\,dW_s,
\qquad
0\leq t\leq T.
\]
The initial condition and the two It\^o characteristics of $Y^{(j)}$
satisfy
\[
|Y_0^{(j)}|\leq R_j,
\qquad
|\pi_{R_j}(a_s)|\leq R_j,
\qquad
\|\pi_{R_j}(H_s)\|\leq R_j
\]
almost surely, with the last two inequalities holding for almost every
$s\in[0,T]$. Consequently, $Y^{(j)}$ satisfies the assumptions of
Lemma~\ref{lem:bounded-functional-quantization} with $G=R_j$. We next verify that
$Y^{(j)}$ agrees with $X$ on the event $B_j$.
On $B_j$, we have
\[
|X_0|\leq R_j,
\qquad
|a_s|\leq R_j,
\qquad
\|H_s\|\leq R_j
\]
for almost every $s\in[0,T]$.
The finite-variation terms in the representations
of $X$ and $Y^{(j)}$ are thus equal on $B_j$. Moreover, the quadratic
variation of the difference of their martingale terms is zero on $B_j$.
It follows that
\[
X=Y^{(j)}
\qquad\text{on }B_j
\]
as elements of $\mathcal C^d$, up to a null set.

\medskip
\noindent\textbf{Step 3. Quantization Error on $B_0,\ldots,B_J$.}
We now choose the codebook sizes $n_j$ for the truncated processes $Y^{(j)}$.
Let
\[
J:=\left\lfloor\frac{\log n}{2}\right\rfloor,
\qquad
n_j:=\left\lfloor \frac{1-e^{-1}}4ne^{-j}\right\rfloor,
\quad 0\leq j\leq J.
\]
For every $0\leq j\leq J$, one has
\[
ne^{-j}
\geq
ne^{-J}
\geq
n^{1/2}.
\]
It follows that, for all sufficiently large $n$,
\begin{equation}
n_j\geq2
\qquad\text{and}\qquad
\log n_j\geq\frac{\log n}{3},\label{eq:endpoint:nj-lower-bound}
\end{equation}
uniformly over $0\leq j\leq J$. The total number of allocated codepoints satisfies
\begin{equation}
1+\sum_{j=0}^J n_j
\leq
1+\frac{1-e^{-1}}4n\sum_{j=0}^\infty e^{-j}
=
1+\frac n4
\leq n,\label{eq:endpoint:codepoint-budget}
\end{equation}
for all $n\geq2$. Apply Lemma~\ref{lem:bounded-functional-quantization} to $Y^{(j)}$, with
moment exponent $\theta q$ and codebook size $n_j$. Since the bounding
constant for the initial condition, drift, and diffusion coefficient is
$R_j$, there exists a deterministic codebook
$A_j\subset\mathcal C^d$ with $|A_j|\leq n_j$ such that
\begin{equation}
\left(
\mathbb E
d_\infty\bigl(Y^{(j)},A_j\bigr)^{\theta q}
\right)^{1/(\theta q)}
\leq
C_{d,\theta q,T}
R_j(\log n_j)^{-1/2}.\label{eq:endpoint:shell-quantization}
\end{equation}
Combining \eqref{eq:endpoint:shell-quantization} with
\eqref{eq:endpoint:nj-lower-bound} gives
\[
\left(
\mathbb E
d_\infty\bigl(Y^{(j)},A_j\bigr)^{\theta q}
\right)^{1/(\theta q)}
\leq
C R_j(\log n)^{-1/2},
\qquad
0\leq j\leq J,
\]
where, throughout the remainder of the proof, $C$ denotes a finite
positive constant independent of $n$ and $j$. Define the union codebook
\[
A:=\{0\}\cup\bigcup_{j=0}^J A_j.
\]
By \eqref{eq:endpoint:codepoint-budget},
$|A|\leq n$. We first estimate the distortion on
$\bigcup_{j=0}^J B_j$. Since
$A_j\subset A$ and $X=Y^{(j)}$ on $B_j$, one has
\[
\begin{aligned}
\mathbb E\left[
d_\infty(X,A)^q\mathbf 1_{B_j}
\right]
&\leq
\mathbb E\left[
d_\infty\bigl(Y^{(j)},A_j\bigr)^q
\mathbf 1_{B_j}
\right].
\end{aligned}
\]
Applying H\"older's inequality with conjugate exponents
$\theta$ and $\theta/(\theta-1)$ gives
\[
\begin{aligned}
\mathbb E\left[
d_\infty(X,A)^q\mathbf 1_{B_j}
\right]
&\leq
\left(
\mathbb E
d_\infty\bigl(Y^{(j)},A_j\bigr)^{\theta q}
\right)^{1/\theta}
\mathbb P(B_j)^\beta \\
&=
\left[
\left(
\mathbb E
d_\infty\bigl(Y^{(j)},A_j\bigr)^{\theta q}
\right)^{1/(\theta q)}
\right]^q
\mathbb P(B_j)^\beta 
\leq
C R_j^q(\log n)^{-q/2}\mathbb P(B_j)^\beta.
\end{aligned}
\]
For $j\geq1$, Markov's inequality and the moment assumption on
$\Lambda_X$ yield
\[
\mathbb P(B_j)
\leq
\mathbb P(\Lambda_X>2^j)
\leq
2^{-\rho j}\mathbb E\Lambda_X^\rho.
\]
Since $R_j=2^{j+1}$, it follows that
\[
\begin{aligned}
\sum_{j=0}^J R_j^q\mathbb P(B_j)^\beta
&\leq
2^q
+
C\sum_{j=1}^J
2^{q(j+1)}2^{-\rho\beta j} 
\leq
C
+
C\sum_{j=1}^\infty
2^{j(q-\rho\beta)}.
\end{aligned}
\]
The final series converges because $\rho\beta>q$. We have consequently
obtained
\begin{equation}
\sum_{j=0}^J
\mathbb E\left[
d_\infty(X,A)^q\mathbf 1_{B_j}
\right]
\leq
C(\log n)^{-q/2}.\label{eq:endpoint:encoded-shell-total}
\end{equation}

\medskip
\noindent\textbf{Step 4. Quantization Error on $\{\Lambda_X>2^{J+1}\}$.}
It remains to estimate the distortion on $\{\Lambda_X>2^{J+1}\}$. Since the
zero path belongs to $A$,
$d_\infty(X,A)\leq\|X\|_\infty.$
Up to a null set, the identity
\[
\Omega\setminus\bigcup_{j=0}^J B_j
=
\{\Lambda_X>2^{J+1}\}
\]
holds. H\"older's inequality with conjugate exponents
$\rho/q$ and $\rho/(\rho-q)$ therefore gives
\begin{equation}
\mathbb E\left[
d_\infty(X,A)^q
\mathbf 1_{\{\Lambda_X>2^{J+1}\}}
\right]
\leq
\mathbb E\left[
\|X\|_\infty^q
\mathbf 1_{\{\Lambda_X>2^{J+1}\}}
\right]
\leq
\left(\mathbb E\|X\|_\infty^\rho\right)^{q/\rho}
\mathbb P(\Lambda_X>2^{J+1})^{1-q/\rho}.\label{eq:endpoint:tail-holder}
\end{equation}
Another application of Markov's inequality gives
\begin{equation}
\mathbb P(\Lambda_X>2^{J+1})
\leq
2^{-\rho(J+1)}
\mathbb E\Lambda_X^\rho.\label{eq:endpoint:tail-probability}
\end{equation}
Substituting \eqref{eq:endpoint:tail-probability} and
\eqref{eq:endpoint:supremum-moment} into \eqref{eq:endpoint:tail-holder} yields
\begin{equation}
\mathbb E\left[
d_\infty(X,A)^q
\mathbf 1_{\{\Lambda_X>2^{J+1}\}}
\right]
\leq
C2^{-(\rho-q)(J+1)}.\label{eq:endpoint:tail-shell-error}
\end{equation}
Because $J=\lfloor(\log n)/2\rfloor$, the right-hand side is
exponentially small in $\log n$. In particular, for all sufficiently
large $n$,
\begin{equation}
2^{-(\rho-q)(J+1)}
\leq
Ce^{-c\log n}
\leq
C(\log n)^{-q/2}.\label{eq:endpoint:tail-rate}
\end{equation}
Combining \eqref{eq:endpoint:encoded-shell-total},
\eqref{eq:endpoint:tail-shell-error}, and \eqref{eq:endpoint:tail-rate} yields
\[
\mathbb E d_\infty(X,A)^q
\leq
C(\log n)^{-q/2}.
\]
Since $|A|\leq n$, the definition of the functional quantization error
gives
\[
e_{n,q}(X)
\leq
\left(
\mathbb E d_\infty(X,A)^q
\right)^{1/q}
\leq
C(\log n)^{-1/2}
\]
for every sufficiently large $n$. For the finitely many remaining values of $n$, use the codebook consisting
only of the zero path. Lyapunov's inequality and
\eqref{eq:endpoint:supremum-moment} give
\[
e_{n,q}(X)
\leq
\left(\mathbb E\|X\|_\infty^q\right)^{1/q}
<\infty.
\]
Increasing the constant to cover these finitely many values completes the
proof.
\end{proof}

\begin{proof}[Proof of Theorem~\ref{thm:functional-quantization}]
Proposition~\ref{prop:endpoint-functional-quantization} gives the upper
bound with the stated dependence of the constant. For the lower bound,
the supremum-norm quantization asymptotics of Dereich and Scheutzow
\cite[Theorem~1.1 and Remark~1.2]{DereichScheutzow2006}, with Hurst
parameter $1/2$, give
\begin{equation}
    e_{n,q}(B)\asymp(\log n)^{-1/2}.
\end{equation}
Brownian scaling extends their result on $[0,1]$ to $[0,T]$.
Embedding scalar codebooks in the first coordinate and projecting
$\R^d$-valued codebooks onto that coordinate show that
$e_{n,q}(\bar B)=e_{n,q}(B)$. The asymptotic lower bound extends to all
$n\geq2$ by decreasing $c_{q,T}$, since $e_{n,q}(B)>0$ for every finite
$n$. Finally, $\bar B$ has zero initial value and drift, and a diffusion
integrand of Hilbert--Schmidt norm one. Thus $\Lambda_{\bar B}=1$, which
proves the asserted optimality over each class with a fixed moment bound.
\end{proof}

\subsubsection{From functional quantization to empirical laws}

Functional quantization alone does not control the empirical measure, because
a deterministic quantizer and an equal-weight empirical sample have different
support and weight structures. The following proposition supplies the missing
transfer. Truncation confines the problem to a bounded ball, projection onto a
finite codebook reduces the central part to multinomial sampling, and a final
balance between bounded fluctuations and tail errors yields the deviation
bound.
\begin{proposition}
\label{prop:quantization-to-empirical}
Let $p\ge1$, and let $Z$ be a $\Cc^d$-valued random variable with law
$\eta$. Assume that, for some $\rho>p$ and $A>0$,
\begin{equation}
    \E\|Z\|_\infty^\rho<\infty,
    \qquad
    e_{n,p}(Z)
    \le
    A(\log n)^{-1/2},
    \qquad n\ge2.
\end{equation}
Let $Z_1,\dots,Z_N$ be independent copies of $Z$, and define
    $\eta_N:=\frac1N\sum_{i=1}^N\delta_{Z_i}.$
Let
\begin{equation}
    \beta_{\rho,p}
    :=
    \min\left\{\frac{\rho}{p}-1,\frac{\rho}{2p}\right\}.
\end{equation}
There exist constants $C_{p,\rho,A,\eta}>0$ and $c_p>0$, where
$C_{p,\rho,A,\eta}$ depends on $\eta$ only through
$\E\|Z\|_\infty^\rho$, such that, for every $N\ge2$ and $t>0$,
\begin{equation}
    \P\left(
        \Wc_p(\eta_N,\eta)>C_{p,\rho,A,\eta}(\log N)^{-1/2}+t
    \right)
    \le
    \exp\left(
        -\frac{c_pNt^{2p}}{(\log N)^{p^2/(\rho-p)}}
    \right)
    +C_{p,\rho,A,\eta}N^{-\beta_{\rho,p}}t^{-\rho}.
    \label{eq:quantization-concentration}
\end{equation}
\end{proposition}

\begin{proof}
Except for constants displayed explicitly, the subscripts on constants
denoted by $C$ or $c$ list all parameters on which they may depend. Their
values may change from line to line, and they are independent of $N$ and
$t$.

\smallskip

\noindent {\bf Step 1. Truncation.}
Fix $N\ge16$ and $t>0$, and set
\begin{equation}
    R:=(\log N)^{p/[2(\rho-p)]}.
\end{equation}
Let $T_R$ be the radial truncation
\begin{equation}
    T_R(z)
    :=
    \begin{cases}
        z, & \|z\|_\infty\le R,\\[3pt]
        \dfrac{R}{\|z\|_\infty}z, & \|z\|_\infty>R.
    \end{cases}
\end{equation}
A direct comparison of the cases according to whether the two norms exceed
$R$ gives
\begin{equation}
    \|T_R(x)-T_R(y)\|_\infty
    \le
    \|x-y\|_\infty
    +\big|\|x\|_\infty-\|y\|_\infty\big|
    \le
    2\|x-y\|_\infty.\label{eq:conc:truncation-lipschitz}
\end{equation}
Set
\begin{equation}
    \eta^R:=(T_R)_\#\eta,
    \qquad
    \eta_N^R:=(T_R)_\#\eta_N.
    \qquad
    \xi_i:=\|Z_i-T_R(Z_i)\|_\infty^p
    =\big(\|Z_i\|_\infty-R\big)_+^p.
\end{equation}
The moment assumption gives
\begin{equation}
    \E\xi_1
    \le
    R^{-(\rho-p)}\E\|Z\|_\infty^\rho,
    \qquad
    \E\xi_1^{\rho/p}
    \le
    \E\|Z\|_\infty^\rho.\label{eq:conc:xi-moments}
\end{equation}
The canonical couplings yield
\begin{equation}
    \Wc_p^p(\eta,\eta^R)
    \le
    \E\xi_1,
    \qquad
    \Wc_p^p(\eta_N,\eta_N^R)
    \le
    \frac1N\sum_{i=1}^N\xi_i. \label{eq:conc:1}
\end{equation}
Since
\begin{equation}
    \E|\xi_1-\E\xi_1|^{\rho/p}
    \le
    C_{p,\rho}\E\|Z\|_\infty^\rho,
\end{equation}
combining two classical results, one from von Bahr and Esseen
\cite[Theorem~2, p.~301]{vonBahrEsseen1965} (for $p<\rho\le2p$) and another from Rosenthal
\cite[Theorem~3, p.~279]{Rosenthal1970} (for
$\rho>2p$), give
\begin{align}
    \E\left|
        \frac1N\sum_{i=1}^N(\xi_i-\E\xi_1)
    \right|^{\rho/p}
    \le
    C_{p,\rho}\E\|Z\|_\infty^\rho
    \begin{cases}
        N^{1-\rho/p}, & p<\rho\le2p,\\[3pt]
        N^{1-\rho/p}+N^{-\rho/(2p)}, & \rho>2p,
    \end{cases}
    \le
    C_{p,\rho}\E\|Z\|_\infty^\rho N^{-\beta_{\rho,p}}.
\end{align}
The implication
\begin{equation}
    \left\{
        \Wc_p(\eta_N,\eta_N^R)>(\E\xi_1)^{1/p}+\frac{t}{2}
    \right\}
    \subset
    \left\{
        \frac1N\sum_{i=1}^N(\xi_i-\E\xi_1)>\left(\frac{t}{2}\right)^p
    \right\}
\end{equation}
follows from $(a+b)^p\ge a^p+b^p$ for $a,b\ge0$. Markov's inequality
therefore yields
\begin{equation}
    \P\left(
        \Wc_p(\eta_N,\eta_N^R)>(\E\xi_1)^{1/p}+\frac{t}{2}
    \right)
    \le
    C_{p,\rho,\eta}N^{-\beta_{\rho,p}}t^{-\rho}. \label{eq:conc:2}
\end{equation}

\smallskip
\noindent {\bf Step 2. Projected quantization.}
Choose a deterministic codebook for $Z$ with at most
$\lfloor N^{1/2}\rfloor$ points and $p$-distortion at most
$2A\big(\log\lfloor N^{1/2}\rfloor\big)^{-1/2}$.
Projecting its points under $T_R$ produces a codebook $\Ac$ with at most
$\lfloor N^{1/2}\rfloor$ points and diameter at most $2R$.
By \eqref{eq:conc:truncation-lipschitz},
\begin{equation}
    \left(
        \E d_\infty(T_R(Z),\Ac)^p
    \right)^{1/p}
    \le
    4A\big(\log\lfloor N^{1/2}\rfloor\big)^{-1/2}
    \le
    C_pA(\log N)^{-1/2}.
\end{equation}
Let $Q:\Cc^d\to\Ac$ be a measurable nearest-point projection, with ties
resolved by a fixed deterministic rule, and set
\begin{equation}
    \bar\eta:=Q_\#\eta^R,
    \qquad
    \bar\eta_N:=Q_\#\eta_N^R.
\end{equation}
The population and samplewise nearest-point couplings give
\begin{equation}
    \Wc_p^p(\eta^R,\bar\eta)
    \le
    \E d_\infty(T_R(Z),\Ac)^p
    \le
    C_pA^p(\log N)^{-p/2},\label{eq:conc:projected-population}
\end{equation}
and
\begin{equation}
    \E\Wc_p^p(\eta_N^R,\bar\eta_N)
    \le
    \frac1N\sum_{i=1}^N
    \E d_\infty(T_R(Z_i),\Ac)^p
    \le
    C_pA^p(\log N)^{-p/2}.\label{eq:conc:projected-sample}
\end{equation}

\smallskip
\noindent {\bf Step 3. Bound between truncated laws.}
For $a\in\Ac$, the definition gives $$\bar\eta_N(\{a\})=N^{-1}\sum_{i=1}^N\mathbf 1_{\{Q(T_R(Z_i))=a\}}.$$ Clearly, $\mathbf 1_{\{Q(T_R(Z_i))=a\}}$ are independent Bernoulli random variables with success probability $\bar\eta(\{a\})$. Hence
\begin{align}
    \E\bar\eta_N(\{a\})
    =\bar\eta(\{a\}),  
    \qquad
    \operatorname{Var}\bigl(\bar\eta_N(\{a\})\bigr)
    =\frac{\bar\eta(\{a\})\bigl(1-\bar\eta(\{a\})\bigr)}{N}.
\end{align}
Since $\operatorname{diam}(\Ac)\le2R$, by the usual Wasserstein-TV inequality,
\begin{align}
    \E\Wc_p^p(\bar\eta_N,\bar\eta)
    &\le
    \frac{(2R)^p}{2}
    \sum_{a\in\Ac}
    \E|\bar\eta_N(\{a\})-\bar\eta(\{a\})| 
    \le
    \frac{(2R)^p}{2\sqrt N}
    \sum_{a\in\Ac}
    \sqrt{\bar\eta(\{a\})(1-\bar\eta(\{a\}))}  \\
    &\le
    \frac{(2R)^p}{2\sqrt N}|\Ac|^{1/2}
    \left(
        \sum_{a\in\Ac}
        \bar\eta(\{a\})(1-\bar\eta(\{a\}))
    \right)^{1/2} 
    \le
    \frac{(2R)^p}{2\sqrt N}N^{1/4}
    \le
    C_pR^pN^{-1/4}.\label{eq:conc:finite-codebook}
\end{align}
The triangle inequality, together with \eqref{eq:conc:projected-population},
\eqref{eq:conc:projected-sample}, and \eqref{eq:conc:finite-codebook}, yields
\begin{align}
    \E\Wc_p^p(\eta_N^R,\eta^R)
    &\le
    3^{p-1}\left[
        \E\Wc_p^p(\eta_N^R,\bar\eta_N)
        +\E\Wc_p^p(\bar\eta_N,\bar\eta)
        +\Wc_p^p(\bar\eta,\eta^R)
    \right]  \\
    &\le
    C_p\left[
        A^p(\log N)^{-p/2}
        +R^pN^{-1/4}
    \right].\label{eq:conc:truncated-mean}
\end{align}
For $\mathbf z=(z_1,\dots,z_N)\in(\Cc^d)^N$ with
$\max_i\|z_i\|_\infty\le R$, set
\begin{equation}
    \Phi(\mathbf z)
    :=
    \Wc_p^p\left(
        \frac1N\sum_{i=1}^N\delta_{z_i},
        \eta^R
    \right).
\end{equation}
Suppose that $\mathbf z$ and $\mathbf z'$ are two such vectors that differ
only in coordinate $j$. Let $(X,Y)$ be an optimal coupling of
$N^{-1}\sum_{i=1}^N\delta_{z_i}$ and $\eta^R$. For each distinct point $x$ among
$z_1,\dots,z_N$, write $I_x:=\{i:z_i=x\}$. Attach to $(X,Y)$ a random
label $I$ by setting
\begin{equation}
    \P(I=i\mid X,Y)
    =
    \frac{\mathbf 1_{\{z_i=X\}}}{|I_X|},
    \qquad i=1,\dots,N.
\end{equation}
In other words, given $X=x$, the label $I$ is chosen
uniformly from $I_x$. Therefore
\begin{equation}
    \P(I=i)
    =
    \E\bigl[\P(I=i\mid X,Y)\bigr]
    =
    \frac{\P(X=z_i)}{|I_{z_i}|}
    =
    \frac{|I_{z_i}|/N}{|I_{z_i}|}
    =
    \frac1N.
\end{equation}
Thus $I$ is uniform on $\{1,\dots,N\}$, and $X=z_I$, a.s.. The optimality of coupling $(X,Y)$ gives
\begin{equation}
    \Phi(\mathbf z)
    =
    \E d_\infty(X,Y)^p
    =
    \E d_\infty(z_I,Y)^p.
\end{equation}
Now set $X':=z_I'$. Since $I$ is uniform, $X'$ has law
$N^{-1}\sum_{i=1}^N\delta_{z_i'}$, while $Y$ still has law $\eta^R$. Hence $(X',Y)$ is a
coupling of these two measures. Moreover, $X'=X$ on $\{I\ne j\}$, whereas
$X=z_j$ and $X'=z_j'$ on $\{I=j\}$. Therefore
\begin{align}
    \Phi(\mathbf z')
    &\le
    \E d_\infty(X',Y)^p
    =
    \E d_\infty(z_I',Y)^p 
    =
    \Phi(\mathbf z)
    +\E\left[
        \mathbf 1_{\{I=j\}}
        \left(d_\infty(z_j',Y)^p-d_\infty(z_j,Y)^p\right)
    \right]  \\
    &=
    \Phi(\mathbf z)
    +\frac1N\E\left[
        d_\infty(z_j',Y)^p-d_\infty(z_j,Y)^p
        \,\middle|\, I=j
    \right] 
    \le
    \Phi(\mathbf z)+\frac{(2R)^p}{N}.
\end{align}
Interchanging $\mathbf z$ and $\mathbf z'$ gives
\begin{equation}
    |\Phi(\mathbf z)-\Phi(\mathbf z')|
    \le
    \frac{(2R)^p}{N}.
\end{equation}
The bounded-difference inequality of Boucheron, Lugosi, and Massart
\cite[Theorem~6.2]{BoucheronLugosiMassart2013}, with coordinate bound
$(2R)^p/N$ and deviation $(t/2)^p$, implies
\begin{align}
    &\P\left(
        \Wc_p(\eta_N^R,\eta^R)
        >\left(\E\Wc_p^p(\eta_N^R,\eta^R)\right)^{1/p}+\frac{t}{2}
    \right)  \\
    &\qquad\le
    \P\left(
        \Wc_p^p(\eta_N^R,\eta^R)
        -\E\Wc_p^p(\eta_N^R,\eta^R)>\left(\frac{t}{2}\right)^p
    \right)  
    \le
    \exp\left(
        -\frac{c_pNt^{2p}}{R^{2p}}
    \right).\label{eq:conc:3}
\end{align}

\smallskip
\noindent {\bf Step 4. Balance the parameters and complete the proof.}
Taking $p$th roots in \eqref{eq:conc:truncated-mean} and using
\eqref{eq:conc:xi-moments} gives
\begin{equation}
    \left(\E\Wc_p^p(\eta_N^R,\eta^R)\right)^{1/p}
    +2(\E\xi_1)^{1/p}
    \le
    C_{p,\rho,A,\eta}\left[
        (\log N)^{-1/2}
        +RN^{-1/(4p)}
        +R^{-(\rho-p)/p}
    \right].
\end{equation}
Here $R^{-(\rho-p)/p}=(\log N)^{-1/2}$, while
\begin{equation}
    (\log N)^{1/2}RN^{-1/(4p)}
    =
    (\log N)^{1/2+p/[2(\rho-p)]}N^{-1/(4p)}
    \le
    C_{p,\rho}
\end{equation}
for $N\ge16$. Noticing \eqref{eq:conc:1}, outside the union of the events in \eqref{eq:conc:2} and \eqref{eq:conc:3}, one has
\begin{align}
    \Wc_p(\eta_N,\eta)
    &\le
    \Wc_p(\eta_N,\eta_N^R)
    +\Wc_p(\eta_N^R,\eta^R)
    +\Wc_p(\eta^R,\eta)  \\
    &\le
    \left(\E\Wc_p^p(\eta_N^R,\eta^R)\right)^{1/p}
    +2(\E\xi_1)^{1/p}+t
    \le
    C_{p,\rho,A,\eta}(\log N)^{-1/2}+t.
\end{align}
Finally, since $R^{2p}=(\log N)^{p^2/(\rho-p)}$, a union bound and the
probability estimates in \eqref{eq:conc:2} and \eqref{eq:conc:3} prove
\eqref{eq:quantization-concentration} for $N\geq16$.

For $2\leq N<16$, coupling both measures through $\delta_{\mathbf 0}$ and
using Jensen's inequality give
$\E\Wc_p(\eta_N,\eta)^\rho\leq C_\rho\E\|Z\|_\infty^\rho$; Markov's
inequality then extends \eqref{eq:quantization-concentration} to this finite
range after increasing $C_{p,\rho,A,\eta}$.
\end{proof}

\begin{corollary}
\label{cor:quantization-to-empirical-mean}
Under the hypotheses of Proposition
\ref{prop:quantization-to-empirical}, there exists
$C_{p,\rho,A,\eta}>0$ such that
\begin{equation}
    \E\big[\Wc_p(\eta_N,\eta)\big]
    \le
    \left(
        \E\big[\Wc_p(\eta_N,\eta)^p\big]
    \right)^{1/p}
    \le
    C_{p,\rho,A,\eta}(\log N)^{-1/2},
    \qquad N\ge2.
\end{equation}
\end{corollary}

\begin{proof}
For $N\geq16$, retain the notation $R$, $\eta^R$, $\eta_N^R$, and $\xi_i$
from the preceding proof. Minkowski's inequality and the canonical couplings
in \eqref{eq:conc:1} give
\begin{align}
    \left(
        \E\Wc_p(\eta_N,\eta)^p
    \right)^{1/p}
    &\le
    \left(
        \E\Wc_p(\eta_N,\eta_N^R)^p
    \right)^{1/p}
    +
    \left(
        \E\Wc_p(\eta_N^R,\eta^R)^p
    \right)^{1/p}
    +\Wc_p(\eta^R,\eta) \\
    &\le
    2(\E\xi_1)^{1/p}
    +
    \left(
        \E\Wc_p(\eta_N^R,\eta^R)^p
    \right)^{1/p} \\
    &\le
    C_{p,\rho,A,\eta}
    \left[
        (\log N)^{-1/2}
        +RN^{-1/(4p)}
        +R^{-(\rho-p)/p}
    \right]
    \le C_{p,\rho,A,\eta}(\log N)^{-1/2}.
\end{align}
The third line follows from \eqref{eq:conc:xi-moments} and
\eqref{eq:conc:truncated-mean}, while the last line follows from the choice of
$R$ and the calculation in Step~4.

For $2\leq N<16$, coupling both measures through $\delta_{\mathbf 0}$ gives
$\bigl(\E\Wc_p(\eta_N,\eta)^p\bigr)^{1/p}
\leq2\bigl(\E\|Z\|_\infty^p\bigr)^{1/p}$, which is absorbed into the
claimed bound after increasing the constant.
\end{proof}

\begin{proof}[Proof of Theorem~\ref{thm:endpoint-path-law}]
The moment estimate \eqref{eq:endpoint:supremum-moment} and Theorem
\ref{thm:functional-quantization}, applied with $q=p$, verify the
hypotheses of Corollary \ref{cor:quantization-to-empirical-mean} with
$Z=X$ and $\eta=\mu$. The dependence of the constants in these three results
gives $C=C\bigl(d,T,p,\rho,\E\Lambda_X^\rho\bigr)$, as claimed.
\end{proof}

\begin{proof}[Proof of Theorem~\ref{cor:polynomial-path-law-concentration}]
The moment estimate \eqref{eq:endpoint:supremum-moment} and Theorem
\ref{thm:functional-quantization}, applied with $q=p$, verify the
hypotheses of Proposition \ref{prop:quantization-to-empirical} with
$Z=X$, $\eta=\mu$, and $\eta_N=\mu_N$. This proves
\eqref{eq:main_concen} with
$C=C\bigl(d,T,p,\rho,\E\Lambda_X^\rho\bigr)$.

Taking $t=(\log N)^{-1/2}$ in \eqref{eq:main_concen} changes the threshold to
$(C+1)(\log N)^{-1/2}$. The exponential term is bounded by
$CN^{-\beta_{\rho,p}}(\log N)^{\rho/2}$ uniformly over $N\geq2$ after
increasing $C$. Absorbing the additional one in the threshold and renaming
the constant proves the remaining assertion.
\end{proof}

\begin{proof}[Proof of Proposition~\ref{prop:transport-concentration}]
The product law $\mu^{\otimes N}$ satisfies the same quadratic transport
inequality for the product metric
\begin{equation}
    \left(
        \sum_{i=1}^N d_\infty(x_i,x_i')^2
    \right)^{1/2}.
\end{equation}
The reverse triangle inequality and the coupling of equally labelled atoms
give
\begin{align}
    &\left|
        \Wc_p\left(\frac1N\sum_{i=1}^N\delta_{x_i},\mu\right)
        -
        \Wc_p\left(\frac1N\sum_{i=1}^N\delta_{x_i'},\mu\right)
    \right|  \\
    &\qquad\le
    \left(
        \frac1N\sum_{i=1}^N d_\infty(x_i,x_i')^p
    \right)^{1/p}
    \le
    N^{-1/\max\{p,2\}}
    \left(
        \sum_{i=1}^N d_\infty(x_i,x_i')^2
    \right)^{1/2}.
\end{align}
The product tensorization theorem and the corresponding mean-centered
concentration inequality of Gozlan and L\'eonard
\cite[Theorem~4.12 and Lemma~6.1]{GozlanLeonard2007}
therefore yield
\begin{equation}
    \P\left(
        \Wc_p(\mu_N,\mu)
        >\E\Wc_p(\mu_N,\mu)+t
    \right)
    \le
    \exp\left(
        -\frac{N^{2/\max\{p,2\}}t^2}{2C_T}
    \right).
\end{equation}
Theorem~\ref{thm:endpoint-path-law} bounds the expectation by
$C\bigl(d,T,p,\rho,\E\Lambda_X^\rho\bigr)(\log N)^{-1/2}$.
\end{proof}

\subsection{Proof of Proposition \ref{prop:arbitrarily-slow-empirical-laws}}\label{sec:slowrate}

\begin{proof}
\noindent\textbf{Step 1. Construction}
Let $e_1$ be the first vector of the canonical basis of $\mathbb R^d$.  For
$j\geq 1$, set
\[
t_j:=T2^{-j},\qquad h_j:=T2^{-j-3},
\]
and define the continuous triangular function
\[
\phi_j(t)=\max\left\{1-\frac{|t-t_j|}{h_j},0\right\},
\qquad 0\leq t\leq T.
\]
The support of $\phi_j$ is the interval $[t_j-h_j,t_j+h_j]$.  
These intervals are pairwise disjoint, they accumulate only at zero, and $\phi_j(t_j)=1$.

Define a decreasing sequence of amplitudes by
$a_j=r_{\max\{1,\lfloor 2^{j-2}\rfloor\}}, j\geq 1.$
Then $0<a_j\leq 1$ and $a_j\to 0$.  Let $(\varepsilon_j)_{j\geq 1}$ be
independent Rademacher random variables and set
\[
X(t)=\sum_{j=1}^{\infty}a_j\varepsilon_j\phi_j(t)e_1,
\qquad 0\leq t\leq T.
\]
Because the supports of the functions $\phi_j$ are disjoint, for every $m\geq 0$
and every choice of signs the tail satisfies
\[
\left\|\sum_{j>m}a_j\varepsilon_j\phi_j e_1\right\|_\infty
=\sup_{j>m}a_j=a_{m+1}.
\]
The series therefore converges uniformly.  Its limit is continuous away from zero,
and continuity at zero follows from $a_j\to 0$.  Thus $X$ is a $C^d$-valued
random variable.

\vspace{0.5em}

\noindent\textbf{Step 2. Compact support}
Consider the set of all possible paths as described in Step 1, i.e.,
\[
K:=\left\{
\sum_{j=1}^{\infty}a_j\sigma_j\phi_j e_1:
\sigma_j\in\{-1,1\}\text{ for every }j
\right\}.
\]
Every element of $K$ has supremum norm at most $a_1\leq 1$.  To verify
compactness, take an arbitrary sequence in $K$.  Since each sign has only two
possible values, a diagonal extraction produces a subsequence and signs
$(\sigma_j)_{j\geq 1}$ such that, for every fixed $m$, the first $m$ signs of the
subsequence eventually agree with $\sigma_1,\ldots,\sigma_m$.  Once this agreement
holds, the supremum distance from the corresponding limiting path is at most
$2a_{m+1}$.  Since $a_{m+1}\to 0$, the extracted subsequence converges uniformly
to an element of $K$.  Hence $K$ is sequentially compact, and therefore compact,
in the metric space $\Cc^d$.  With
\[
\mu=\mathcal L(X),
\]
the measure $\mu$ is supported on $K$.

\vspace{0.5em}

\noindent\textbf{Step 3. Lower Bound}
We next prove the lower bound.  Fix an integer $m\geq 1$.  For each
$\sigma=(\sigma_1,\ldots,\sigma_m)\in\{-1,1\}^m$, let $K_\sigma$ be the cylinder consisting of paths whose first $m$ signs equal $\sigma$.  The sets
$\{K_\sigma\}_{\sigma\in \{-1,+1\}^m}$ form a partition of $K$, and independence of the Rademacher variables
gives
\[
\mu(K_\sigma)=2^{-m}.
\]
If $\sigma\neq\tau$, choose an index $j\leq m$ at which the two sign vectors
differ.  For any $x\in K_\sigma$ and $y\in K_\tau$, evaluation at the peak $t_j$
gives
\begin{align}\label{eq:distance-lowerbound}
\|x-y\|_\infty
\geq |x(t_j)-y(t_j)|
=2a_j
\geq 2a_m.
\end{align}

Let $\nu$ be supported on the points $y_1,\ldots,y_q$, where $q\leq N$.  Call a
cylinder $K_\sigma$ close to $y_i$ when
\[
\inf_{x\in K_\sigma}\|x-y_i\|_\infty<a_m.
\]
By \eqref{eq:distance-lowerbound}, any point $y_i$ can be close to at most one cylinder.  
Consequently, at most $N$ of the $2^m$ cylinders are close to the support of $\nu$.

Choose
$m=\left\lceil\log_2(2N)\right\rceil.$
Let $U$ be the union of the cylinders that are not close to any point in the support
of $\nu$.  Since $2^m\geq 2N$, the equal cylinder masses imply
\[
\mu(U)\geq 1-\frac{N}{2^m}\geq\frac12.
\]
Moreover, every point of $U$ is at distance at least $a_m$ from the support of
$\nu$.  Therefore, for every coupling $\pi\in\Pi(\mu,\nu)$,
\[
\int_{C^d\times C^d}\|x-y\|_\infty^p\,\pi(dx,dy)
\geq a_m^p\pi(U\times C^d)
=a_m^p\mu(U)
\geq\frac{a_m^p}{2}.
\]
Taking the infimum over all couplings yields
\[
\Wc_p(\mu,\nu)\geq 2^{-1/p}a_m.
\]
The choice of $m$ gives $2^{m-2}<N$, and hence
\[
\max\{1,\lfloor 2^{m-2}\rfloor\}\leq N.
\]
Since $(r_N)$ is nonincreasing, it follows that $a_m\geq r_N$.  This proves the
first assertion.  Applying it to each realization of $\mu_N$, which is supported
on at most $N$ points, proves the samplewise empirical lower bound.

\end{proof}

\subsection{Proof of Proposition \ref{prop:SDE}}
\label{subsec:proof-sde-application}
\begin{proof}
Since \(X\) is continuous and adapted, the \(\Cc^d\)-valued process
\(t\mapsto X_{t\wedge\cdot}\) is also adapted. Moreover, for almost every
\(\omega\), the map
\[
t\longmapsto X_{t\wedge\cdot}(\omega)
\]
is continuous with respect to the uniform norm on \(\Cc^d\). Hence
the stopped-path process is progressively measurable. Since \(b\) and
\(\sigma\) are Borel measurable, it follows that
\[
a_t=b(t,X_{t\wedge\cdot}),
\qquad
H_t=\sigma(t,X_{t\wedge\cdot})
\]
are progressively measurable. The equation in Assumption~\ref{ass:SDE} can therefore be written as
\begin{equation}
X_t
=
X_0+\int_0^t a_s\,ds+\int_0^t H_s\,dW_s,
\qquad 0\leq t\leq T.\label{eq:sde:semimartingale}
\end{equation}
It remains to verify the required integrability of the time-essential
suprema. 
Define
$X_t^*:=\sup_{0\leq u\leq t}|X_u|.$
We first prove that
$\mathbb E\bigl[(X_T^*)^\rho\bigr]<\infty.$

For \(n\in\mathbb N\), let
\[
\tau_n
:=
\inf\{t\in[0,T]:X_t^*\geq n\}\wedge T,
\qquad
X_t^{*,n}
:=
\sup_{0\leq u\leq t}|X_{u\wedge\tau_n}|.
\]
For \(s\leq \tau_n\), the linear growth assumption gives
\[
|a_s|+\|H_s\|
\leq
L\bigl(1+\|X_{s\wedge\cdot}\|_\infty\bigr)
=
L(1+X_s^*)
=
L(1+X_s^{*,n}).
\]
Consequently,
\[
\mathbf 1_{\{s\leq\tau_n\}}
\bigl(|a_s|+\|H_s\|\bigr)
\leq
L(1+X_s^{*,n}).
\]
Using the stopped form of \eqref{eq:sde:semimartingale}, the elementary
inequality
$|x+y+z|^\rho\leq C_\rho\bigl(|x|^\rho+|y|^\rho+|z|^\rho\bigr),$
and the Burkholder--Davis--Gundy inequality, we obtain, for
\(0\leq t\leq T\),
\begin{align}
\mathbb E\bigl[(X_t^{*,n})^\rho\bigr]
&\leq
C_\rho\mathbb E|X_0|^\rho
+
C_\rho\mathbb E
\left(
\int_0^t
\mathbf 1_{\{s\leq\tau_n\}}|a_s|\,ds
\right)^\rho
+
C_\rho\mathbb E
\left[
\sup_{0\leq u\leq t}
\left|
\int_0^u
\mathbf 1_{\{s\leq\tau_n\}}H_s\,dW_s
\right|^\rho
\right]
\notag\\
&\leq
C_\rho\mathbb E|X_0|^\rho
+
C_{\rho,L}\mathbb E
\left(
\int_0^t(1+X_s^{*,n})\,ds
\right)^\rho
+
C_{\rho,L}\mathbb E
\left(
\int_0^t(1+X_s^{*,n})^2\,ds
\right)^{\rho/2}.\label{eq:sde:stopped-moment-bdg}
\end{align}
By Hölder's inequality,
\begin{equation}
\left(
\int_0^t(1+X_s^{*,n})\,ds
\right)^\rho
\leq
t^{\rho-1}
\int_0^t(1+X_s^{*,n})^\rho\,ds.\label{eq:sde:drift-holder}
\end{equation}
Furthermore,
\begin{align*}
\left(
\int_0^t(1+X_s^{*,n})^2\,ds
\right)^{\rho/2}
&\leq
\left[
(1+X_t^{*,n})
\int_0^t(1+X_s^{*,n})\,ds
\right]^{\rho/2}.
\end{align*}
Thus, by Young's inequality, for every \(\varepsilon>0\),
\begin{align}
\left(
\int_0^t(1+X_s^{*,n})^2\,ds
\right)^{\rho/2}
&\leq
\varepsilon(1+X_t^{*,n})^\rho
+
C_{\varepsilon,\rho}
\left(
\int_0^t(1+X_s^{*,n})\,ds
\right)^\rho
\notag\\
&\leq
\varepsilon(1+X_t^{*,n})^\rho
+
C_{\varepsilon,\rho,T}
\int_0^t(1+X_s^{*,n})^\rho\,ds.\label{eq:sde:diffusion-young}
\end{align}
Combining \eqref{eq:sde:stopped-moment-bdg},
\eqref{eq:sde:drift-holder}, and \eqref{eq:sde:diffusion-young} with the inequality
$(1+x)^\rho\leq 2^{\rho-1}(1+x^\rho), ~x\geq0,$
gives
\begin{align*}
\mathbb E\bigl[(X_t^{*,n})^\rho\bigr]
\leq
C_{\rho,L,T}\bigl(1+\mathbb E|X_0|^\rho\bigr)
+
C_{\rho,L}\varepsilon
\mathbb E\bigl[(X_t^{*,n})^\rho\bigr]
+
C_{\varepsilon,\rho,L,T}
\int_0^t
\mathbb E\bigl[(X_s^{*,n})^\rho\bigr]\,ds.
\end{align*}
Choose \(\varepsilon>0\) sufficiently small so that the second term on
the right-hand side can be absorbed into the left-hand side. We then
obtain
\[
\mathbb E\bigl[(X_t^{*,n})^\rho\bigr]
\leq
C_{\rho,L,T}\bigl(1+\mathbb E|X_0|^\rho\bigr)
+
C_{\rho,L,T}
\int_0^t
\mathbb E\bigl[(X_s^{*,n})^\rho\bigr]\,ds.
\]
Grönwall's lemma yields
\[
\sup_{n\in\mathbb N}
\mathbb E\bigl[(X_T^{*,n})^\rho\bigr]
\leq
C_{\rho,L,T}\bigl(1+\mathbb E|X_0|^\rho\bigr)
<\infty.
\]
Since \(X\) has continuous paths on the compact interval \([0,T]\),
we have \(X_T^*<\infty\) almost surely and
$X_T^{*,n}\longrightarrow X_T^*$
a.s.
Fatou's lemma therefore gives
\[
\mathbb E\bigl[(X_T^*)^\rho\bigr]
\leq
\liminf_{n\to\infty}
\mathbb E\bigl[(X_T^{*,n})^\rho\bigr]
\leq
C_{\rho,L,T}\bigl(1+\mathbb E|X_0|^\rho\bigr).
\]
Finally, the linear growth condition implies, for every \(s\in[0,T]\),
\[
|a_s|+\|H_s\|
\leq
L\bigl(1+\|X_{s\wedge\cdot}\|_\infty\bigr)
=
L(1+X_s^*)
\leq
L(1+X_T^*).
\]
Hence
\[
\operatorname*{ess\,sup}_{0\leq s\leq T}|a_s|
\leq L(1+X_T^*),
\qquad
\operatorname*{ess\,sup}_{0\leq s\leq T}\|H_s\|
\leq L(1+X_T^*).
\]
Since \(L\geq1\) and \(|X_0|\leq X_T^*\), it follows that
$\Lambda_X
\leq
L(1+X_T^*).$
Therefore,
\[
\mathbb E\Lambda_X^\rho
\leq
L^\rho 2^{\rho-1}
\left(
1+\mathbb E\bigl[(X_T^*)^\rho\bigr]
\right)
\leq
C_{\rho,L,T}\bigl(1+\mathbb E|X_0|^\rho\bigr).
\]
Thus all the requirements of
Assumption~\ref{assum:regularity} are satisfied. The result and the stated
dependence of the constant now follow from
Theorem~\ref{thm:endpoint-path-law}.
\end{proof}

\subsection{Proof of Proposition \ref{thm:app}}
\label{subsec:proof-mckean-vlasov-application}
The proof has two components. The first verifies that the McKean-Vlasov limit equation
satisfies the quantization estimate. The second couples each
interacting particle with a conditional independent copy driven by the same
initial condition and Brownian motion, and then propagates the i.i.d.\ 
empirical error through the Lipschitz dynamics.

\subsubsection{Quantization of the limit dynamics}

\begin{corollary}
\label{coro:conditional_copy}
Let $p\geq1$, and let $Z$ be a $\Cc^d$-valued random variable. Assume that,
for some $\rho>p$ and $A>0$,
\begin{equation}
\E\norm{Z}_\infty^\rho<\infty,
\qquad
e_{n,p}(Z)
\leq
A(\log n)^{-1/2},
\qquad
n\geq2.
\label{eq:conditional-copy-assumption}
\end{equation}
Let $\Gc\subset\Fc$ be a sub-$\sigma$-field, and define
\[
\eta^{\Gc}:=\Lc(Z\mid\Gc).
\]
Let $Z_1,\ldots,Z_N$ be conditionally independent given $\Gc$ and suppose
that
\[
\Lc(Z_i\mid\Gc)=\eta^{\Gc},
\qquad
i=1,\ldots,N.
\]
Set
\[
\eta_N
:=
\frac1N\sum_{i=1}^N\delta_{Z_i}.
\]
There exists a constant $C_{p,\rho,A,\eta}>0$, where
$C_{p,\rho,A,\eta}$ depends on the unconditional law
$\eta:=\Lc(Z)$ only through $\E\norm{Z}_\infty^\rho$, such that
\begin{equation}
\E\big[\Wc_p(\eta_N,\eta^{\Gc})\big]
\leq
\left(
    \E\big[\Wc_p(\eta_N,\eta^{\Gc})^p\big]
\right)^{1/p}
\leq
C_{p,\rho,A,\eta}(\log N)^{-1/2},
\qquad
N\geq2.
\label{eq:conditional-copy-conclusion}
\end{equation}
\end{corollary}
\begin{proof}
The proof follows the same argument as Corollary
\ref{cor:quantization-to-empirical-mean}, conditioning throughout on
$\Gc$. We only indicate the modification in Step~3 of Proposition
\ref{prop:quantization-to-empirical}, where independence of the sample is
used.

Let $\bar\eta^{\Gc}$ denote the conditional law of
$Q(T_R(Z_1))$ given $\Gc$. For every $a\in\mathcal A$,
\[
    \bar\eta_N(\{a\})
    =
    \frac1N\sum_{i=1}^N
    \mathbf 1_{\{Q(T_R(Z_i))=a\}}.
\]
Conditionally on $\Gc$, the random variables
\[
    \mathbf 1_{\{Q(T_R(Z_i))=a\}},
    \qquad i=1,\dots,N,
\]
are independent Bernoulli random variables with conditional success
probability $\bar\eta^{\Gc}(\{a\})$. Hence
\[
    \E\left[
        \bar\eta_N(\{a\})
        \,\middle|\,
        \Gc
    \right]
    =
    \bar\eta^{\Gc}(\{a\}),
    \qquad
    \operatorname{Var}\left(
        \bar\eta_N(\{a\})
        \,\middle|\,
        \Gc
    \right)
    =
    \frac{
        \bar\eta^{\Gc}(\{a\})
        \bigl(1-\bar\eta^{\Gc}(\{a\})\bigr)
    }{N}.
\]
Therefore the estimate in Step~3 holds conditionally on $\Gc$ with
$\bar\eta$ replaced by $\bar\eta^{\Gc}$. Taking expectations afterwards,
the remainder of the proof is unchanged.
\end{proof}

\begin{lemma}
\label{lem:nonlinear-quantization}
Under Assumption~\ref{ass:coefficients},
$\E\norm{Y}_\infty^r<\infty$, and there exists a constant
$C>0$ such that
\[
e_{n,p}(Y)\leq C(\log n)^{-1/2},
\qquad n\geq2.
\]
\end{lemma}

\begin{proof}
The linear-growth estimate in Assumption~\ref{ass:coefficients}, together
with the elementary inequality $(a+b+c)^r\leq C_r(a^r+b^r+c^r)$, gives
\[
\abs{b(t,x,\nu)}^r+\norm{\sigma(t,x,\nu)}^r+\norm{\sigma_0(t,x,\nu)}^r
\leq
C\bigl(
    1+\norm{x}_\infty^r
    +\Wc_p(\nu,\delta_{\mathbf 0})^r
\bigr).
\]
The standard localization argument used in the proof of Proposition
\ref{prop:SDE}, with H\"older's inequality for the drift term,
the Burkholder--Davis--Gundy inequality for the martingale term, and Young's
inequality when $r<2$, gives
\[
\E\norm{Y_{t\wedge\cdot}}_\infty^r
\leq
C\bigl(1+\E\abs{\xi}^{r}\bigr)
+
C\int_0^t
\left(
    \E\norm{Y_{s\wedge\cdot}}_\infty^r
    +\E[\Wc_p(\mu_s,\delta_{\mathbf 0})^r]
\right)\dd s.
\]
Because $r>p$ and $\mu_s=\Lc(Y_{s\wedge\cdot}|W^0)$,
\[
\Wc_p(\mu_s,\delta_{\mathbf 0})^r
=
\left(
    \E\norm{Y_{s\wedge\cdot}}_\infty^p
|W^0\right)^{r/p}
\leq
\E[\norm{Y_{s\wedge\cdot}}_\infty^r|W^0].
\]
Taking expectation and Gronwall's inequality therefore shows that
\[
\E\norm{Y}_{\infty}^{r}
\leq
C\bigl(1+\E\abs{\xi}^{r}\bigr).
\]
Let $\Gc:=\Fc_T^{W^0}$ and choose
$\mu_t=(S_t)_\#\mu$, where $\mu=\Lc(Y\mid\Gc)$ and
$S_t(x):=x_{t\wedge\cdot}$.
By nonanticipativity of the strong solution and independence
of future common-noise increments from the joint past,
\[
\mu_t
=
\Lc(Y_{t\wedge\cdot}\mid\Fc_t^{W^0}).
\]
Thus $(\mu_t)$ is adapted to the common-noise filtration.
The preceding moment bound and dominated convergence under
$\mu$ show that $t\mapsto\mu_t$ is almost surely continuous
in $\Wc_p$.

Set
\[
a_t:=b(t,Y_{t\wedge\cdot},\mu_t),
\qquad
H_t:=
\bigl[
\sigma(t,Y_{t\wedge\cdot},\mu_t)
\ \sigma_0(t,Y_{t\wedge\cdot},\mu_t)
\bigr].
\]
Since the stopped-path process and the conditional-law flow
are continuous and adapted, the Borel measurability of the
coefficients implies that $a$ and $H$ are progressively
measurable. Hence $Y$ has the required It\^o representation
with driving Brownian motion $(W,W^0)$.

Since $S_t$ is a contraction fixing the zero path,
$\Wc_p(\mu_t,\delta_{\mathbf0})\leq \Wc_p(\mu,\delta_{\mathbf0})$ for every $t$.
Conditional Jensen's inequality gives
\[
\E \Wc_p(\mu,\delta_{\mathbf0})^r
=
\E\left[
\left(\E[\norm{Y}_\infty^p\mid\Gc]\right)^{r/p}
\right]
\leq
\E\norm{Y}_\infty^r.
\]
The linear-growth estimate therefore yields
\[
\Lambda_Y
\leq C(1+\norm{Y}_\infty+\Wc_p(\mu,\delta_{\mathbf0})),
\qquad
\E\Lambda_Y^r
\leq C\bigl(1+\E\norm{Y}_\infty^r\bigr)
<\infty.
\]
Thus Assumption~\ref{assum:regularity} holds with $\rho=r$.
Theorem~\ref{thm:functional-quantization}, applied with $q=p$,
proves the quantization estimate.
\end{proof}

\subsubsection{Synchronous-coupling stability}

We now transfer the empirical approximation of the law of the limit process to the interacting system. The coupling separates the total error into a conditionally i.i.d.\ empirical term and a dynamical stability term. The latter is estimated in $L^p$; the treatment of the martingale integrals differs according to whether $p\geq2$ or $p<2$, but both regimes close by Gr\"onwall's inequality.

\begin{proof}[Proof of Proposition~\ref{thm:app}]
For every $i\in\{1,\ldots,N\}$, let $Y^i$ solve the limit equation with
the same initial value and the same Brownian motion as the $i$th interacting
particle, i.e., 
\[
Y_t^i
=
\xi^i
+
\int_0^t b(s,Y_{s\wedge\cdot}^i,\mu_s)\dd s
+
\int_0^t \sigma(s,Y_{s\wedge\cdot}^i,\mu_s)\dd W_s^i
+
\int_0^t \sigma_0(s,Y_{s\wedge\cdot}^i,\mu_s)\dd W_s^0.
\]
Because the stopped-path law flow $(\mu_s)_{0\leq s\leq T}$ is
adapted to the common noise filtration and the pairs $(\xi^i,W^i)$ are independent, the processes
$Y^1,\ldots,Y^N$ are conditionally independent copies of $Y$ given $W^0$. Introduce their empirical full path measure and empirical stopped-path
measure by
\[
\nu^N
:=
\frac1N\sum_{i=1}^N\delta_{Y^i},
\qquad
\nu_s^N
:=
\frac1N\sum_{i=1}^N\delta_{Y_{s\wedge\cdot}^i}.
\]
Lemma~\ref{lem:nonlinear-quantization} and Corollary
\ref{coro:conditional_copy}, applied with $\Gc = \Fc^{W^0}_T$,$\rho=r$, show that
\[
a_N
:=
\left(
\E\Wc_p(\nu^N,\mu)^p
\right)^{1/p}
\leq
C(\log N)^{-1/2}.
\]
Set
$\Delta_t^i:=X_t^{i,N}-Y_t^i.$
Subtracting the two stochastic equations gives
\[
\begin{aligned}
\Delta_t^i
=
\int_0^t
\bigl[
 b(s,X_{s\wedge\cdot}^{i,N},\mu_s^N)
 -b(s,Y_{s\wedge\cdot}^i,\mu_s)
\bigr]\dd s
&+
\int_0^t
\bigl[
 \sigma(s,X_{s\wedge\cdot}^{i,N},\mu_s^N)
 -\sigma(s,Y_{s\wedge\cdot}^i,\mu_s)
\bigr]\dd W_s^i
\\
&+
\int_0^t
\bigl[
 \sigma_0(s,X_{s\wedge\cdot}^{i,N},\mu_s^N)
 -\sigma_0(s,Y_{s\wedge\cdot}^i,\mu_s)
\bigr]\dd W_s^0.
\end{aligned}
\]
By Assumption~\ref{ass:coefficients}, the absolute value of the drift
difference and the Hilbert--Schmidt norm of the diffusion difference are
both bounded, up to the factor $L$, by
\[
\norm{X_{s\wedge\cdot}^{i,N}-Y_{s\wedge\cdot}^i}_\infty
+\Wc_p(\mu_s^N,\mu_s)
=
\sup_{0\leq u\leq s}\abs{\Delta_u^i}
+\Wc_p(\mu_s^N,\mu_s).
\]
Synchronous coupling gives
\[
\Wc_p(\mu_s^N,\nu_s^N)^p
\leq
\frac1N\sum_{j=1}^N
\sup_{0\leq u\leq s}\abs{\Delta_u^j}^p.
\]
The contraction property of Wasserstein distance yields
$
\Wc_p(\nu_s^N,\mu_s)
\leq
\Wc_p(\nu^N,\mu).
$
Consequently, we have that for every $t\leq T$,
\begin{equation}
\Wc_p(\mu_s^N,\mu_s)^p
\leq
C\left[
    \frac1N\sum_{j=1}^N
    \sup_{0\leq u\leq s}\abs{\Delta_u^j}^p
    +\Wc_p(\nu^N,\mu)^p
\right],\label{eq:app:path-law-pointwise}
\end{equation}
and
\begin{equation}
\sup_{0\leq s\leq t}\Wc_p(\mu_s^N,\mu_s)^p
\leq
C\left[
    \frac1N\sum_{j=1}^N
    \sup_{0\leq u\leq t}\abs{\Delta_u^j}^p
    +\Wc_p(\nu^N,\mu)^p
\right].\label{eq:app:path-law-uniform}
\end{equation}
Write
\[
U_N(t)
:=
\frac1N\sum_{i=1}^N
\E\sup_{0\leq u\leq t}\abs{\Delta_u^i}^p.
\]
To proceed with the proof, we discuss two seperate cases by the value of $p$ below.
\medskip

\noindent\textbf{Case $p\geq2$.}
H\"older's inequality for the time integral and the
Burkholder--Davis--Gundy inequality give, for $t\leq T$,
\[
\E\sup_{0\leq u\leq t}\abs{\Delta_u^i}^p
\leq
C\int_0^t
\E\left[
\sup_{0\leq u\leq s}\abs{\Delta_u^i}^p
+\Wc_p(\mu_s^N,\mu_s)^p
\right]\dd s.
\]
Averaging over $i$ and using \eqref{eq:app:path-law-uniform} gives
\[
U_N(t)
\leq
C\int_0^t\bigl(U_N(s)+a_N^p\bigr)\dd s.
\]
Gronwall's inequality therefore yields
$U_N(T)\leq Ca_N^p$.

\medskip
\noindent\textbf{Case $p\in[1,2)$.}
For notational convenience, set
\[
h_s^i
:=
\sup_{0\leq u\leq s}\abs{\Delta_u^i}
+\Wc_p(\mu_s^N,\mu_s).
\]
H\"older's inequality for the drift integral and the
Burkholder--Davis--Gundy inequality give
\[
\E\sup_{0\leq u\leq t}\abs{\Delta_u^i}^p
\leq
C\int_0^t\E(h_s^i)^p\dd s
+
C\E\left(\int_0^t(h_s^i)^2\dd s\right)^{p/2}.
\]
Since $2-p>0$, one has
\[
\left(\int_0^t(h_s^i)^2\dd s\right)^{p/2}
\leq
\left(\sup_{0\leq u\leq t}(h_u^i)^p\right)^{(2-p)/2}
\left(\int_0^t(h_s^i)^p\dd s\right)^{p/2}.
\]
Young's inequality with conjugate exponents $2/(2-p)$ and $2/p$ shows that,
for every $\varepsilon>0$,
\[
\left(\int_0^t(h_s^i)^2\dd s\right)^{p/2}
\leq
\varepsilon\sup_{0\leq u\leq t}(h_u^i)^p
+
C_{p,\varepsilon}\int_0^t(h_s^i)^p\dd s.
\]
It follows that
\begin{equation}
\E\sup_{0\leq u\leq t}\abs{\Delta_u^i}^p
\leq
\varepsilon\E\sup_{0\leq u\leq t}(h_u^i)^p
+
C_{\varepsilon}\int_0^t\E(h_s^i)^p\dd s,
\label{eq:app:stability-p-lt2}
\end{equation}
after rescaling $\varepsilon$ to absorb the multiplicative constant.

Equations \eqref{eq:app:path-law-uniform} and
\eqref{eq:app:path-law-pointwise} imply, respectively, for every $s\leq t$,
\begin{equation}\label{eq:app:h-uniform}
\frac1N\sum_{i=1}^N
\E\sup_{0\leq s\leq t}(h_s^i)^p
\leq
C\bigl(U_N(t)+a_N^p\bigr),
\quad
\frac1N\sum_{i=1}^N\E(h_s^i)^p
\leq
C\bigl(U_N(s)+a_N^p\bigr).
\end{equation}
Averaging \eqref{eq:app:stability-p-lt2} over $i$ and using
\eqref{eq:app:h-uniform} gives
\[
U_N(t)
\leq
C\varepsilon\bigl(U_N(t)+a_N^p\bigr)
+
C_{\varepsilon}\int_0^t
\bigl(U_N(s)+a_N^p\bigr)\dd s.
\]
Choose $\varepsilon>0$ sufficiently small that $C\varepsilon\leq1/2$.
The term involving $U_N(t)$ can then be absorbed into the left-hand side,
and Gronwall's inequality yields
\[
U_N(T)\leq Ca_N^p.
\]
Together with the bound $U_N(T)\leq Ca_N^p$ established for $p\geq2$,
the preceding estimate proves the same bound for every $p\geq1$. Synchronous coupling thus gives
\[
\left(
\E\Wc_p(\mu^N,\nu^N)^p
\right)^{1/p}
\leq
U_N(T)^{1/p}
\leq
Ca_N.
\]
Finally, the triangle inequality for $\Wc_p$ and Minkowski's inequality imply
\[
\begin{aligned}
\left(
\E\Wc_p(\mu^N,\mu)^p
\right)^{1/p}
\leq
\left(
\E\Wc_p(\mu^N,\nu^N)^p
\right)^{1/p}
+
\left(
\E\Wc_p(\nu^N,\mu)^p
\right)^{1/p}  
\leq
Ca_N
\leq
C(\log N)^{-1/2}.
\end{aligned}
\]
The stated mean estimate follows from H\"older's inequality.
\end{proof}

\end{document}